\documentclass[a4paper,twoside,10pt]{article}
\usepackage{amsmath,amssymb,graphics,epsfig,color}
\usepackage{amsmath,amsthm,amsfonts,latexsym,amscd,amssymb,enumerate,wasysym,stmaryrd}

\usepackage{graphicx}
\usepackage[all]{xy}
\usepackage{fancyhdr}
\usepackage[T1]{fontenc}
\usepackage{ifthen}
\usepackage{amsmath,amsfonts,amssymb}
\newcommand{\pleinepage}%
{\setlength{\oddsidemargin}{0in}\setlength{\textwidth}{6.26in}\setlength{\topmargin}{0in}\setlength{\textheight}{8.7in}}
\newcommand{\grossepage}%
{\setlength{\oddsidemargin}{-0.5cm}\setlength{\textwidth}{17.5cm}\setlength{\topmargin}{-1.5cm}\setlength{\textheight}{24cm}}
\newcommand{\defit}[1]%
{{\em #1}}

\def\Re{\mathop{\plainRe\mkern -2mu\mit e}\nolimits}
\def\Im{\mathop{\plainIm\mkern -2mu\mit m}\nolimits}
\def\surl#1_#2{\mathrel{\mathop{\kern 0pt #1}\limits_{#2}}}

\newcommand{\fleche}[1]%
{\rTo^{#1}}
\newcommand{\fonction}[5]%
{\begin{diagram}
#2 & {} &\rTo^{#1} & {} & #3 \\
#4 & {} &\rMapsto & {} & #5 
\end{diagram} }
\newcommand{\sfonction}[5]%
{$\begin{array}{ccc}#2 & {\buildrel #1 \over \rightarrow} & #3 \\#4 & \mapsto & #5 \\ \end{array}$ }
\newcommand{\accolade}[1]%
{\begin{cases}  #1 \end{cases}}

\newcounter{nbre}
\newcommand{\entete}[6]%
{{\large \noindent%
\mbox{\begin{tabular}{c} #1 \\ #2 \end{tabular}}\hspace{\fill}\mbox{\begin{tabular}{c} #3 \\ #4 \end{tabular}}\vspace{1cm}\begin{center}{\Huge \textsc{#5}} \\ \vspace{0.5cm}\begin{tabular}{c} #6 \\ \hline \end{tabular}\end{center}\bigskip}}
\newcommand{\defifont}{ \sc }

{\refstepcounter{compteur} \par \vspace{0.3cm}\ \\ \noindent {{  \textsc{\textbf{Definitions}}} 
    \textbf{\thecompteur} \ 
    }---\ \sffamily\renewcommand{\em}{\normalfont\itshape}}{\par
    \vspace{0.3 cm}}

\newcommand{\gloss}[1]%
 {\index{{#1}@{#1}}{\em #1}\relax}
\newcommand{\xgloss}[2]%
 {\index{{#1}@{#1}!{#2}@{#2}}{\em #1\relax #2}\relax}
 \newcommand{\glossref}[2]%
 {\index{{#2}@{#1}}{\em #1}\relax}
 \newcommand{\xglossref}[4]%
 {\index{{#3}@{#1}!{#4}@{#2}}{\em #1 \relax #2 }\relax}
\newcounter{compteur}
\renewcommand{\thecompteur}{\thesection.\arabic{compteur}}
\newenvironment{dfn}[1][]%
{\refstepcounter{compteur} \par \vspace{0.3cm}\ \\ \noindent {{  \textsc{\textbf{Definition}}} 
    \textbf{\thecompteur} \ 
    }---\ \sffamily\renewcommand{\em}{\normalfont\itshape}}{\par
    \vspace{0.3 cm}}
{\refstepcounter{compteur} \par \vspace{0.3cm}\ \\ \noindent {{  \textsc{\textbf{Definitions}}} 
    \thecompteur \ 
    }---\ \sffamily\renewcommand{\em}{\normalfont\itshape}\begin{enumerate}}{\end{enumerate}
    \par \vspace{0.5 cm}}
\newcounter{theonum}

\newenvironment{thm}[1][]%
{\refstepcounter{theonum} \par \vspace{0.3cm}\ \\ \noindent {{  \textsc{\textbf{Theorem}}} 
    \textbf{\thecompteur} \ 
    }---\ \sffamily\renewcommand{\em}{\normalfont\itshape}}{\par
    \vspace{0.3 cm}}

\newcommand\addpage[2]{#2, page #1}
\makeatletter
\renewcommand\p@theonum{\protect\addpage{\thepage}}
\makeatother
 
\newenvironment{prop}[1][]%
{\refstepcounter{compteur} \par \vspace{0.2cm}\ \\ \noindent {{  \textsc{\textbf{Proposition}}} 
    \textbf{\thecompteur} \ 
    }---\ \sffamily\renewcommand{\em}{\normalfont\itshape}}{\par
    \vspace{0.2 cm}}

\newenvironment{prop*}[1][]%
{ \par \vspace{0.2cm}\ \\ \noindent {{\textsc{\textbf{Proposition}}} 
    #1 \ 
    }---\sffamily\renewcommand{\em}{\normalfont\itshape}}{\par \vspace{0.3 cm}}
{\refstepcounter{compteur} \par \vspace{0.3cm}\ \\ \noindent {{  \textsc{\textbf{Properties}}} 
    \thecompteur \ 
    }---\ \sffamily\renewcommand{\em}{\normalfont\itshape}}{\par
    \vspace{0.3 cm}}

\newenvironment{cor}[1][]%
{\refstepcounter{compteur} \par \vspace{0.3cm}\ \\ \noindent {{ \textsc{\textbf{Corollary}}}
    \textbf{\thecompteur} 
    \ }---\ \sffamily\renewcommand{\em}{\normalfont\itshape}}{\par  \vspace{0.3 cm}}

\newenvironment{lem}[1][]%
{\refstepcounter{compteur} \par \vspace{0.2cm}\ \\ \noindent {{  \textsc{\textbf{Lemma}}} 
    \textbf{\thecompteur} \ 
    }---\ \sffamily\renewcommand{\em}{\normalfont\itshape}}{\par
    \vspace{0.2 cm}}    
    
\renewenvironment{proof}%
{\par \vspace{0.2cm}\ \\ \noindent{ { \textsc{Proof}}\,---\ } }{\hfill{$\Box$} \par \vspace{0.2 cm}}

{\par \vspace{0.2cm}\ \\ \noindent{\sc \textbf{Remark}\ }---\ }{\par \vspace{0.2cm}}
{\refstepcounter{compteur}\par \vspace{0.2cm}\ \\ \noindent{\sc \textbf{Conjecture} \textbf{\thecompteur}\ }---\ }{\par \vspace{0.2cm}}

\begin{document}

\NoCompileMatrices
\def\ds{\displaystyle}
\def\pn{\pi_{n}}
\def\pnu{\pi_{n-1}}
\renewcommand{\sectionmark}[1]{\markright{\thesection\ #1}}
\fancyhf{} 
\fancyhead[LE,RO]{\;\thepage}


\renewcommand{\headrulewidth}{0.16pt}
\renewcommand{\footrulewidth}{0pt}
\addtolength{\headheight}{0.7pt} 
\fancypagestyle{plain}{%
\fancyhead{} 
\renewcommand{\headrulewidth}{0pt} 
}

\newcommand{\foot}[1]{\footnote{\begin{normalsize}#1\end{normalsize}}}


\def\bX{\partial X}
\def\dim{\mathop{\rm dim}}
\def\Re{\mathop{\rm Re}}
\def\Im{\mathop{\rm Im}}
\def\I{\mathop{\rm I}}
\def\Id{\mathop{\rm Id}}
\def\grad{\mathop{\rm grad}}
\def\vol{\mathop{\rm vol}}
\def\SU{\mathop{\rm SU}}
\def\SO{\mathop{\rm SO}}
\def\Aut{\mathop{\rm Aut}}
\def\End{\mathop{\rm End}}
\def\GL{\mathop{\rm GL}}
\def\Cinf{\mathop{\mathcal C^{\infty}}}
\def\Ker{\mathop{\rm Ker}}
\def\Coker{\mathop{\rm Coker}}
\def\dom{\mathop{\rm Dom}}
\def\Hom{\mathop{\rm Hom}}
\def\Ch{\mathop{\rm Ch}}
\def\sign{\mathop{\rm sign}}
\def\SF{\mathop{\rm SF}}
\def\loc{\mathop{\rm loc}}
\def\AS{\mathop{\rm AS}}
\def\spec{\mathop{\rm spec}}
\def\Ric{\mathop{\rm Ric}}
\def\ch{\mathop{\rm ch}}
\def\Ch{\mathop{\rm Ch}}

\def\ev{\mathop{\rm ev}}
\def\id\textrm{Id}
\def\dd{\mathcal{D}(d)}
\def\Cli{\mathbb{C}l(1)}
\def\kerd{\operatorname{ker}(d)}

\def\Fi{\Phi}

\def\de{\delta}
\def \dl{\partial L_x^0}
\def\e{\eta}
\def\ep{\epsilon}
\def\ro{\rho}
\def\a{\alpha}
\def\o{\omega}
\def\O{\Omega}
\def\b{\beta}
\def\la{\lambda}
\def\th{\theta}
\def\s{\sigma}
\def\t{\tau}
\def\g{\gamma}
\def\D{\Delta}
\def\G{\Gamma}
\def \fol{\mathcal F}
\def\R{\mathbin{\mathbb R}}
\def\Rn{\R^{n}}
\def\C{\mathbb{C}}
\def\Cm{\mathbb{C}^{m}}
\def\Cn{\mathbb{C}^{n}}
\def\gr{\mathcal{G}}
\def\Kahler{{K\"ahler}}
\def\w{{\mathchoice{\,{\scriptstyle\wedge}\,}{{\scriptstyle\wedge}}
{{\scriptscriptstyle\wedge}}{{\scriptscriptstyle\wedge}}}}
\def\cA{{\cal A}}\def\cL{{\cal L}}
\def\cO{{\cal O}}\def\cT{{\cal T}}\def\cU{{\cal U}}
\def\cD{{\cal D}}\def\cF{{\cal F}}\def\cP{{\cal P}}\def\cH{{\cal H}}\def\cL{{\cal L}}
\def\cB{{\cal B}}


\newcommand{\n}[1]{\left\| #1\right\|}

\def\Z{\mathbb{Z}}
\def\cgs{C^{*}(\Gamma,\sigma)}
\def\bcgs{C^{*}(\Gamma,\bar{\sigma})}
\def\cgsr{C^{*}_{red}(,\sigma)}
\def\Mt{\tilde{M}}
\def\Et{\tilde{E}}
\def\Vt{\tilde{V}}
\def\Xt{\tilde{X}}
\def\N{\mathbb{N}}
\def\Nbs{\N^{\bar{\s}}}
\def\rcab{\ro^{[c]}_{\a-\b}}
\def\rc{\ro^{[c]}}
\def\Cd{\mathbb{C}^{d}}
\def\tr{\mathop{\rm tr}}\def\tralg{\tr{}^{\text{alg}}}       

\def\TR{\mathop{\rm TR}}\def\trace{\mathop{\rm trace}}
\def\STR{\mathop{\rm STR}}
\def\trG{\mathop{\rm tr_\Gamma}}
\def\TRG{\mathop{\rm TR_\Gamma}}
\def\Tr{\mathop{\rm Tr}}
\def\Str{\mathop{\rm Str}}
\def\Cl{\mathop{\rm Cl}}
\def\Op{\mathop{\rm Op}}
\def\supp{\mathop{\rm supp}}
\def\scal{\mathop{\rm scal}}
\def\ind{\mathop{\rm ind}}
\def\Ind{\mathop{\mathcal I\rm nd}\,}
\def\Diff{\mathop{\rm Diff}}
\def\T{\mathcal{T}}
\def\dn{\textrm{dim}_{\Lambda}}
\def \lke{\textrm L^2-\textrm{Ker}}


\newcommand{\wt}{\widetilde}
\newcommand{\go}{\mathcal{G}^{0}}
\newcommand{\dii}{(d_x^{k-1})^\ast}
\newcommand{\di}{d_x^{k-1}}
\newcommand{\ra}{\operatorname{range}}
\newcommand{\rb}{\rangle}
\newcommand{\lb}{\langle}
\newcommand{\re}{\mathcal{R}}
\newcommand{\vo}{\operatorname{End}_{\Lambda}(E)}
\newcommand{\mt}{\mu_{\Lambda,T}}
\newcommand{\tru}{\operatorname{tr}_{\Lambda}}
\newcommand{\buno}{B^1_{\Lambda}(E)}
\newcommand{\bdue}{B^2_{\Lambda}(E)}
\newcommand{\clis}{H^{2k}_{(2),dR}(L_x^0)}
\newcommand{\cali}{L^2(\Omega^{2k}(\partial L_x^0))}
\newcommand{\binf}{B^{\infty}_{\Lambda}(E)}
\newcommand{\bif}{B^{f}_{\Lambda}(E)}
\newcommand{\vn}{\operatorname{End}_{\mathcal{R}}}
\newcommand{\ho}{\operatorname{Hom}_{\Lambda}}
\newcommand{\spc}{\operatorname{spec}_{\Lambda,e}}
\newcommand{\ix}{\operatorname{Ind}_{\Lambda}}
\newcommand{\cic}{C^{\infty}_c(L_x;E_{|L_x})}
\newcommand{\ci}{C^{\infty}_c(L_x;E_{|L_x})}
\newcommand{\tx}{\{T_x\}_{x\in X}}
\newcommand{\cc}{C^{\infty}_c(X)}
\newcommand{\rom}{\underline{\mathcal{R}_0}}
\newcommand{\roma}{(\mathcal{R}_0)_{|\partial X_0}     }
\newcommand{\dfo}{D^{\mathcal{F}_{\partial}}}
\newcommand{\deu}{D_{\epsilon,u}}
\newcommand{\deupp}{D^{+}_{\epsilon,u}}
\newcommand{\deum}{D^{-}_{\epsilon,u}}
\newcommand{\deuf}{D_{\epsilon,u}^{\mathcal{F}_{\partial}}}
\newcommand{\deufo}{D_{\epsilon,u,x_0}^{\mathcal{F}_{\partial}}}
\newcommand{\pie}{\Pi_{\epsilon}}
\newcommand{\pal}{\partial L_x}
\newcommand{\pr}{\partial_r}
\newcommand{\inbl}{\int_{\partial L_x} }
\newcommand{\pkp}{\chi_{\{0\}}(D^+_x)}
\newcommand{\deup}{D_{\epsilon,x}^{\pm}}
\newcommand{\dext}{D_{\epsilon,\mp u,x}^{\pm}}
\newcommand{\dex}{D_{\epsilon,\pm u,x}^{\pm}}
\newcommand{\ext}{\operatorname{Ext}(D_{\epsilon,x}^{\pm})}
\newcommand{\eppu}{0<|u|<\epsilon}
\newcommand{\hdeupx}{e^{-tD^2_{\epsilon,u,x}}}
\newcommand{\udif}{\operatorname{UDiff}}
\newcommand{\ki}{L^2(\Omega^kL_x^0)}
\newcommand{\uc}{\operatorname{UC}}
\newcommand{\op}{\operatorname{Op}}
\newcommand{\deux}{D_{\epsilon,u,x}}
\newcommand{\pk}{\phi_k}
\newcommand{\hdeupsx}{e^{-sD^2_{\epsilon,u,x}}}
\newcommand{\hdeups}{e^{-sD^2_{\epsilon,u}}}
\newcommand{\hdeut}{e^{-tD^2_{\epsilon,u}}}
\newcommand{\indu}{\operatorname{ind}_{\Lambda}}
\newcommand{\stru}{\operatorname{str}_{\Lambda}}
\newcommand{\deuq}{D^2_{\epsilon,u}}
\newcommand{\intk}{\int_{\sqrt{k}}^{\infty}}
\newcommand{\dmd}{d\mu_{\Lambda,D_{\epsilon,u}}(x)}
\newcommand{\defox}{D_{x}^{\mathcal{F}_{\partial}}}
\newcommand{\mun}{\mu_{\Lambda,D_{\epsilon,u}}(x)}
\newcommand{\tsi}{\int_{-\sigma}^{\sigma}}
\newcommand{\ak}{\lim_{k\rightarrow \infty}\operatorname{LIM}_{s\rightarrow 0}}
\newcommand{\deus}{D_{\epsilon,u}e^{-tD_{\epsilon,u}^2}}

\newcommand{\deuss}{D_{\epsilon,u}^2e^{-tD_{\epsilon,u}^2}}
\newcommand{\pkd}{\phi_k^2}
\newcommand{\eup}{e^{-tD^{+}_{\epsilon,u}D^{-}_{\epsilon,u}}}
\newcommand{\eum}{e^{-tD^{-}_{\epsilon,u} D^{+}_{\epsilon,u} }}
\newcommand{\deussx}{D_{\epsilon,u,x}e^{-tD_{\epsilon,u,x}^2}}
\newcommand{\dessx}{S_{\epsilon,u,x}e^{-tS_{\epsilon,u,x}^2}}
\newcommand{\clib}{c(\partial_r)\partial_r \phi_k^2}
\newcommand{\sk}{\int_s^{\sqrt{k}}}
\newcommand{\esm}{S_{\epsilon,u}e^{-tS_{\epsilon,u}^2}}
\newcommand{\deussxo}{D_{\epsilon,u,z_0}e^{-tD_{\epsilon,u,x_0}^2}}
\newcommand{\dessxo}{S_{\epsilon,u,z_0}e^{-tS_{\epsilon,u,z_0}^2}}
\newcommand{\deusszo}{D^{\mathcal{F}_{\partial}}_{\epsilon,u,z_0}e^{-t(D^{\mathcal{F}_{\partial}}_{\epsilon,u,x_0})2}}
\newcommand{\desszo}{S_{\epsilon,u,x_0}e^{-tS_{\epsilon,u,x_0}^2}}
\newcommand{\essp}{S_{\epsilon,u}^2}
\newcommand{\esspo}{S_{0,u}^2}
\newcommand{\dotto}{\dot{\theta}}
\newcommand{\piep}{\Pi_{\epsilon}}
\newcommand{\ome}{\Omega}
\newcommand{\deffo}{D^{\mathcal{F}_{\partial}}}
\newcommand{\nablal}{\nabla_x^l}
\newcommand{\nablak}{\nabla_y^k}
\newcommand{\kerk}{[f(P)]_{(x_0,\bullet)} }
\newcommand{\kepp}{\operatorname{Ker} (D^{\mathcal{F}_0^+})}
\newcommand{\ty}{\infty}
\definecolor{light}{gray}{.95}
\newcommand{\pecetta}[1]{
$\phantom .$
\bigskip
\par\noindent
\colorbox{light}{\begin{minipage}{13.5 cm}#1\end{minipage}}
\bigskip
\par\noindent
}

\newcommand\Di{D\kern-7pt/}

\title{The Atiyah Patodi Singer index formula for measured foliations}
\author{\Large Paolo Antonini\\
Paolo.Antonini@mathematik.uni-regensburg.de\\
paolo.anton@gmail.com}

\maketitle























\begin{abstract}
Let $X_0$ be a compact Riemannian manifold with boundary endowed with a oriented, measured even dimensional foliation with purely transverse boundary. Let $X$ be the manifold with cylinder attached and extended foliation. We prove that the $L^2$--measured index of a Dirac type operator is well defined and the following Atiyah Patodi Singer index formula is true
$$\operatorname{ind}_{L^2,\Lambda}(D^+)=\langle\widehat{A}(X,\nabla)\operatorname{Ch}(E/S),C_{\Lambda}\rangle +1/2[\eta_{\Lambda}(D^{\mathcal{F}_{\partial}})-h^+_{\Lambda}+h^{-}_{\Lambda}].$$ Here $\Lambda$ is a holonomy invariant transverse measure,  $\eta_{\Lambda}(D^{\mathcal{F}_{\partial}})$ is the Ramachandran eta invariant \cite{Rama} of the leafwise boundary operator and the $\Lambda$--dimensions $h^{\pm}_{\Lambda}$ of the space of the limiting values of extended solutions is suitably defined using square integrable representations of the equivalence relation of the foliation with values on weighted Sobolev spaces on the leaves.

\end{abstract}

\section{Introduction}Let $X_0$ be an even dimensional oriented compact Riemannian manifold with boundary equipped with a unitary Clifford module $E\longrightarrow X_0$ with compatibile Clifford connection. Suppose each geometric structure is product type near the boundary. It is a well known fact since the seminal paper by Atiyah Patodi and Singer \cite{AtPaSi1} that the index problem for the Dirac operator $D$ in $X_0$ can be approached in at least two ways;

\noindent 1. a generalized boundary value problem with pseudodifferential boundary condition (the Atiyah Patodi Singer boundary condition)

\noindent 2. an $L^2$ index problem on the manifold $X$ obtained attaching a cylinder to $X_0$ across its boundary.
Actually there is a third completely independent point of view, that of Melrose's $b$ geometry \cite{Me}. This can be seen to correspond to a compactification of $X$ joining a boundary at the infinity.

\noindent  Indeed the operator splits near the boundary as $D=\sigma(D_0+\partial_r)$ where $\sigma$ is a bundle isomorphism, $D_0$ is a Dirac operator on the boundary and $\partial_r$ is the normal derivative. 
Call $\widetilde{D}$ the naturally extended operator on $X$. One can show that $\widetilde{D}$ is Fredholm if and only if the boundary operator is invertible in $L^2$ \cite{Me} but the kernels $\operatorname{Ker}_{L^2}(\widetilde{D}^\pm)$ are finite dimensional and the difference of these dimensions is called the $L^2$ index of $\widetilde{D}$. 
The Atiyah Patodi Singer formula computes this index in terms of the Atiyah Singer local integrand, the eta invariant $\eta(D_0)$ of the boundary operator  and some correcting numbers related to the spaces of the $L^2$--extended solutions on the cylinder,
$$\operatorname{ind}_{L^2}(D^+)=\int_{X_0}\widehat{A}(X_0,\nabla)\operatorname{Ch}(E/S)+\dfrac{1}{2}[\eta(D_0)+h^--h^+].$$
\noindent If $\partial X_0$ has no boundary and is foliated by a smooth foliation equipped with a holonomy invariant transverse measure, Alain Connes \cite{Cos} has generalized, in the contest of non commutative geometry, the Atiyah Singer index formula for a leafwise Dirac operator on $X_0$ i.e. a family of Dirac operators one for each leaf that vary transversally in a measurable way. This result can be seen as a generalization of the Atiyah $L^2$--index theorem for Galois coverings $\Gamma-\widetilde{X_0}\longrightarrow X_0,$ where the Von Neumann algebra associated to the right regular representation of the deck group $\Gamma$ is used to define the $L^2$--index of the lifted Dirac operator on the total space $\widetilde{X_0}.$ In spite of geometrical applications (the signature operatore and the signature formula) one can ask about the existence of an Atiyah Patodi Singer index formula for a foliated manifold with boundary with foliation transverse\footnote{in order to have an induced foliation on the boundary} (normal) to the boundary and a holonomy invariant transverse measure $\Lambda$. This formula has to reflect both the structure of the formula of Connes and Atiyah Patodi Singer. Mohan Ramachandran \cite{Rama} partially solved the problem, proving the A.P.S. theorem for measured foliations for the boundary value problem with A.P.S. boundary condition. In this paper we adopt the second point of view proving the index formula for foliations of manifolds with cylindrical ends. We work, as Ramachandran at level of the leaves using the equivalence relation to prove our main result \begin{thm}
The Dirac operator has finite dimensional $L^2-\Lambda$--index and the following formula holds
$$
\operatorname{ind}_{L^2,\Lambda}(D^+)=\langle\widehat{A}(X)\operatorname{Ch}(E/S),C_{\Lambda}\rangle +1/2[\eta_{\Lambda}(D^{\mathcal{F}_{\partial}})-h^+_{\Lambda}+h^{-}_{\Lambda}]$$ where
$
h^{\pm}_{\Lambda}:=\operatorname{dim}_{\Lambda}(\operatorname{Ext}(D^{\pm})-\operatorname{dim}_{\Lambda}(\operatorname{Ker}_{L^2}(D^{\pm})$ 
are the suitably defined $\Lambda$--dimensions of the space of extended solutions, $C_{\Lambda}$ is the Ruelle--Sullivan current associated to the transverse measure $\Lambda$ and $\eta_{\Lambda}(D^{\mathcal{F}_{\partial}})$ is the Ramachandran eta invariant of the boundary leafwise operator.
\end{thm}
\noindent Our proof is a generalization of a method of Boris Vaillant \cite{Vai} that proved the Atiyah Patodi Singer index formula for Galois coverings of manifolds with cylindrical ends. 
In section 2 we introduce the geometric settings and notations. In section 3 we review the classical Atiyah Patodi Singer index formula paying attention to its cylindrical $L^2$--version. In section 4 
 we show that one can define a version of the essential spectrum of the operator relatively to the Von Neumann algebra of the foliation. Since this spectrum is stable by $\Lambda$--compact perturbations one can carry to foliations the result that the Fredholmness in the sense of Breuer \cite{Br2} is equivalent to boundary invertibility.
 More precisely a \emph{splitting principle} is valid: the Von Neumann essential spectrum is determined by the leafwise behaviour outside compact sets.
  In section 5 we show that the measured $L^2$ index of the Dirac operator is finite. This is our form of elliptic regularity. More precisely we show that the longitudinal measures corresponding to the projections on the kernel and the cokernel are finite on sets of the 
 form (compact in the boundary) $\times$ (cylinder). The construction of the trace as the integration of longitudinal measures against transverse measures makes the rest. 
  Then we perform a two parameter perturbation making the operator Breuer--Fredholm. This can be done, thanks to the splitting principle, by a modification of the boundary operator. In the proof of the index formula we will carefully examine the behaviour of the $\Lambda$--dimensions of the corresponding spaces of the extended solutions of the perturbed operator showing that they converge to the dimensions of the non perturbed one. Each leaf is a manifold with bounded geometry and cylindrical ends. The perturbed operator has product structure displaying a regularizing operator in the cross section of the cylinder times an order one operator with finite propagation speed in the cylindrical direction. In section 6 the cylindrical finite propagation speed is exploited, as in \cite{Vai} to prove some Cheeger Gromov Taylor estimates that relates the heat kernel of the perturbed operator with the non perturbed one.
In section 7 we discuss the existence and convergence properties of the Ramachandran eta invariant for the perturbation. Finally in section 8 the index formula is proven and, in section 9 is shown to be compatible with that of Ramachandran. 

\noindent The author wishes to thanks its thesis advisor Paolo Piazza for having suggested him the problem and for a lot of interesting discussions.
  
\begin{small}
\tableofcontents
\end{small}
 
\section{Geometric Setting}\label{geom}
\noindent A $p$--dimensional foliation $\mathcal{F}$ on a $n$--dimensional manifold with boundary $ X_0$ is transverse to the boundary if it is given by a foliated atlas $\{U_{\alpha}\}$ with homeomorphisms $\phi_{\alpha}:U_{\alpha}\longrightarrow V_{\alpha}\times W_{\alpha}$ with $V_{\alpha}$ open in ${\mathbb{H}}^p:=\{(x_1,...,x_p)\in \R^p:x_1\geq 0\}$ and $W^{q}$ open in $\R^q$ with change of coordinated $\phi_{\alpha}(u,v)$ of the form
 \begin{equation}\label{9009}v'=\phi(v,w),\quad w'=\psi(w)\end{equation} ($\psi$ is a local diffeomorphism).
 Such an atlas is assumed to be maximal among all collections of this type. 
 The integer $p$ is the dimension of the foliation, $q$ its codimension and $p+q=n.$
 In each foliated chart, the connected components of subsets as $\phi_{\alpha}^{-1}(V_{\alpha}\times \{w\})$ are called \underline{plaques}. The plaques coalesce (thanks to the change of coordinate condition \eqref{9009}) to give maximal connected injectively immersed (not embedded) submanifolds called \underline{leaves}. One uses the notation $\mathcal{F}$ for the set of leaves. Note that in general each leaf passes infinitely times through a foliated chart so a foliation is only locally a fibration.
Taking the tangent spaces to the leaves one gets an integrable subbundle $T\mathcal{F}\subset TX_0$ that's transverse to the boundary i.e $T\partial X_0 + T\mathcal{F}=TX_0$ in other words the boundary is a submanifold that's transverse to the foliation.
\noindent Let given on $X$ a smooth oriented foliation $\mathcal{F}$ with leaves of dimension $2p$ respecting the cylindrical structure i.e. 
\begin{enumerate}
\item The submanifold $\partial X_0$ is transversal to the foliation and inherits a foliation with the same codimension $\mathcal{F}_{\partial}=\mathcal{F}_{|\partial X_0}$ with foliated atlas given by $\phi_{\alpha}:U_{\alpha}\cap \partial X_0 \longrightarrow \partial V_{\alpha} \times W_{\alpha}$. 
\item The restriction of the foliation on the cylinder is product type $\mathcal{F}_{|Z}=\mathcal{F}_{\partial}\times [0,\infty).$\end{enumerate}
\noindent Note that these conditions imply that the foliation is \underline{normal} to the boundary.
\noindent The orientation we choose is the one given by $(e_1,..,e_{2p-1},\partial_r)$ where $(e_1,..,e_{2p-1})$ is a positive leafwise frame for the induced boundary foliation. As explained in \cite{AtPaSi1} 
this will choose one of the two possible boundary Dirac operators.
\noindent Let $E\longrightarrow X$ be a leafwise Clifford bundle with leafwise Clifford connection $\nabla^E$ and Hermitian metric $h^E$. Suppose each geometric structure is of product type on the cylinder meaning that if $\rho:\partial X_0\times [0,\infty) \longrightarrow \partial X_0$ is the base projection
$$E_{|Z}\simeq \rho^*(E_{|\partial X_0}),\quad h^E_{|\partial X_0}=\rho^*(h^{E}_{|\partial X_0}),\quad \nabla^E_{|Z}=\rho^*(\nabla^{E}_{|\partial X_0}).$$
\noindent Each geometric object restricts to the leaves to give a longitudinal Clifford module that's canonically $\mathbb{Z}_2$--graded by the leafwise chirality element. One can check immediately that the positive and negative boundary eigenbundles $E^+_{\partial X_0}$ and $E^-_{\partial X_0}$ are both modules for the Clifford structure of the boundary foliation \cite{Me}.
\noindent Leafwise Clifford multiplication by $\partial_r$ induces an isomorphism of  Clifford modules $c(\partial_r):E_{\partial X_0}^+\longrightarrow E_{\partial X_0}^-.$ Put $F=E^+_{|\partial X_0}$ the whole Clifford module on the cylinder $E_{|Z}$ can be identified with the pullback $\rho^*(F\oplus F)$ with the following action:
tangent vectors to the boundary foliation $v\in T\mathcal{F}_{\partial}$ acts as
 $c^{E}(v)\simeq c^F(v)\Omega$ with $\Omega=\left(\begin{array}{cc}0 & 1 \\1 & 0\end{array}\right)$ while in the cylindrical direction $c^F(\partial_r)\simeq \left(\begin{array}{cc}0 & -1 \\1 & 0\end{array}\right)$. 
\noindent In particular one can form the longitudinal Dirac operator assuming under the above identification the form\begin{equation}
\label{diracco}
D=c(\partial_r)\partial_r+c_{|\mathcal F_0}\nabla^{E_{|\mathcal{F}_{\partial}}}=
c(\partial_r)\partial_r+\Omega D^{\mathcal{F}_{\partial}}=
c(-\partial_r)[-\partial_r-c(-\partial_r)\Omega D^{\mathcal{F}_{\partial}}].\end{equation}
Here 
$D^{\mathcal{F}_{\partial}}$ is the leafwise Dirac operator on the boundary foliation.  
\noindent In the following, these identifications will be omitted letting $D$ act directly on $F\oplus F$ according to
\begin{align}\nonumber
&\left(\begin{array}{cc}0 & D^- \\D^+ & 0\end{array}\right)=\left(\begin{array}{cc}0 & -\partial_r+D^{\mathcal{F}_{\partial}} \\ 
 \partial_r+D^{\mathcal{F}_{\partial}}
& 0\end{array}\right)=\left(\begin{array}{cc}0 & \partial_u+D^{\mathcal{F}_{\partial}} \\ 
 -\partial_u+D^{\mathcal{F}_{\partial}}
& 0\end{array}\right)
\end{align} where $u=-r$, $\partial_u=-\partial_r$ (interior unit normal)
note this is the opposite of A.P.S. notation.
\noindent We shall use the notation $X=X_k\cup Z_k$ with $Z_k=\partial X_0\times [k,\infty)$ and $X_k=X_0\cup (\partial X_0\times [0,k])$ also $Z_a^b:=\partial X_{0}\times [a,b]$ and where there's no danger of confusion $Z_x$ is the cylinder of the leaf passing through $x$, $Z_x=L_x\cap Z_0.$

\section{The Atiyah Patodi Singer index theorem}\label{aaps}
\noindent We are going to recall the classical Atiyah--Patodi--Singer index theorem in \cite{AtPaSi1}.
So let $X_0$ be a compact $2p$ dimensional manifold with boundary $\partial X_0$ and consider a Clifford bundle $E$ with all the geometric structure  as in the previous section. We take here the opposite orientation of A.P.S i.e. we use the exterior unit normal to induce the boundary operator instead of the interior one. In other words 
$D^{+}_{\textrm{here}}=D^{-}_{\textrm{APS}}.$ 
The operator writes in a collar around the boundary $
\left(\begin{array}{cc}0 &D^- \\ 
 D_+
& 0\end{array}\right)=
\left(\begin{array}{cc}0 & -\partial_r+D_0 \\ 
 \partial_r+D_0
& 0\end{array}\right)$ where $\partial_r$ is the exterior unit normal and $D_0$ is a Dirac operator on the boundary. It is shown in \cite{AB} that the $K$--theory of the boundary manifold contains topological obstructions to the existence of elliptic boundary value conditions of local type (for the signature operator they are always non zero). If one enlarges the point of view to admit global boundary conditions a Fredholm problem (with Calderon projection \cite{lesch})
 can properly set up. More precisely, consider the boundary operator $D_0$ acting on $\partial X_0$. This is a first order elliptic differential operator with real discrete spectrum on $L^2(\partial X_0;F)$. Let $P=\chi_{[0,\ty)}(D_0)$ be its pseudo differential spectral projection on the non negative part of the spectrum. Then
 
\noindent {\bf{1}.} The (unbounded) operator $D^+:C^{\ty}(X;E^+,P)\longrightarrow C^{\ty}(X,E^-)$ with domain 

$C^{\ty}(X;E^+,P):=\{s\in C^{\ty}(X;E^+):P(s_{|\partial X_0})=0\}$ is Fredholm and the index is given by the formula
$$\operatorname{ind}_{\operatorname{APS}}(D^+)=\int_{X_0} \widehat{A}(X,\nabla )\operatorname{Ch}(E,\nabla)-h/2+\eta(0)/2$$ with the Atiyah--Singer 
$\widehat{A}(X,\nabla)$ 
differential form\footnote{due to the presence of the boundary one does not have here a cohomological pairing, for this reason the notation $\widehat{A}(X,\nabla)$ stresses the dependence from the metric through the connection $\nabla$}, the twisted Chern character $\operatorname{Ch}(E,\nabla)$ \cite{BeGeVe,Me} and two correcting terms:
\begin{itemize}
\item $h:=\operatorname{Ker}(D_0)$ is the dimension of the kernel of the boundary operator
\item $\eta(0)$, the eta invariant of $D_0$ is a spectral invariant which 
gives a measure of the asymmetry of its spectrum. This is extensively explained in section \ref{eta}.
\end{itemize}

\noindent {\bf{2.}} The index formula can be interpreted as a natural $L^2$ problem on the  manifold with a cylinder attached $X$ and every structure pulled back. 
\noindent More precisely the kernel of \\ \noindent $D^{+}:C^{\ty}(X;E^+,P)\longrightarrow C^{\ty}(X,E^-)$ turns out to be naturally isomorphic to the kernel of $D^{+}$ extended to an ubounded operator on $L^2(X)$ while to describe the kernel of its Hilbert space adjoint i.e. the closure of $D^-$ with the adjoint boundary condition
$D^{-}:C^{\ty}(X;E^-,1-P)\longrightarrow C^{\ty}(X,E^+)$ the space of \emph{extended} $L^2$ \emph{solutions} must be introduced.

\noindent A locally square integrable solution $s$ of the equation $D^-s=0$ on $X$ is called an extended solution if for large positive $r$ the equation
\begin{equation}\label{ex2}
s(y,r)=g(y,r)+s_{\ty}(y)\end{equation} is satisfied where $y$ is the coordinate on the base $\partial X_0$ and $g\in L^2$ while $s_{\ty}$ solves $D_0 s_{\ty}=0$ and is called \underline{the limiting value} of $s$.

\noindent APS prove that the kernel of $(D^+)^*$ (Hilbert space adjoint of $D^{+}$ with domain $C^{\ty}(X;E^+,P)$) is naturally isomorphic to the space of $L^2$ extended solution of $D^-$ on $X$.
Moreover \begin{equation}\label{ldueindex}\operatorname{ind}_{\textrm{APS}}(D^+)=\operatorname{dim}_{L^2}(D^+)-\operatorname{dim}_{L^2}(D^-)-h_{\ty}(D^{-})=\operatorname{ind}_{L^2}(D^+)-h_{\ty}(D^{-})\end{equation} where $\operatorname{ind}_{L^2}(D^+):=\operatorname{dim}_{L^2}(D^+)-\operatorname{dim}_{L^2}(D^-)$
and
 the number $h_{\ty}(D^{-})$ is the dimension of the space of limiting values of the extended solutions of $D^-$.
The number at right in \eqref{ldueindex} is sometimes called the $L^2$ \underline{extended index}. Along the proof of \eqref{ldueindex} the authors prove that
\begin{equation}\label{hsp}h=h_{\ty}(D^+)+h_{\ty}(D^-)\end{equation} 
and conjecture that it must be true at level of the kernel of $D_0$ i.e. 

\noindent \emph{every section in} 
$\operatorname{Ker}(D_0)$ 
\emph{is uniquely expressible as a sum of limiting values coming from} 
$D^+$ 
\emph{and} 
$D^-$. 

\noindent The conjecture was solved by Melrose with the invention of the $b$--calculus, a pseudo--differential calculus on a compactification of $X$ that furnished a totally new point of view on the APS problem \cite{Me}. 

\noindent With \eqref{ldueindex} and \eqref{hsp} the index formula is
$$\operatorname{ind}_{L^2}(D^+)=\int_{X_0} \widehat{A}(X_0)\operatorname{Ch}(E)+\dfrac{\eta(0)}{2}+\dfrac{h_{\ty}(D^-)-h_{\ty}(D^+)}{2}.$$

\noindent Finally a naive remark on the nature of the extended solutions in order to motivate our definition of their Von Neumann counterparts $h_{\infty}(D^{\pm})$ (equation \eqref{deffh} and \eqref{1001}). For a real parameter $u$ say that a distributional section $s$ on the cylinder is in the weighted $L^2$--space 
$e^{ur}L^2(\partial X_0\times [0,\ty);E^{\pm})$ if $e^{-ur}s\in L^2$. The operator $D^{\pm}$ trivially esxtends to act on each weighted space. Now it is evident from \eqref{ex2} that an $L^2$--extended solution of the equation $D^+s=0$ is in each $e^{ur}L^2$ for positive $u$. Viceversa let $s\in \bigcap_{u>0}\operatorname{Ker}_{e^{ur}L^2}(D^+)$. Keep $u$ fixed, then $e^{-ur}s\in L^2$ can be represented in terms of a complete eigenfunction expansion for the boundary operator $D_0$,
$e^{-ur}s=\sum_{\lambda}\phi_{\lambda}(y)g(r).$ Solving $D^+s=0$ together with the condition $e^{-ur}s\in L^2$ leads to the representation (on the cylinder)
$s(y,r)=\sum_{\lambda >-u}\phi_{\lambda}(y)g_{0\lambda}(y)e^{-\lambda r}.$ Since $u$ is arbitrary we see that $s$ should have a representation as a sum
$s(y,r)=\sum_{\lambda \geq 0}\phi_{\lambda}(y)g_{0\lambda}e^{-\lambda r}$
over the non negative eigenvalues of $D_0$, i.e. $s$ is 
an extended solution with limiting value $\sum_{\lambda=0}\phi_0(y)g_{00}.$
 We have proved that
$$\operatorname{Ext}(D^{\pm})=\bigcap_{u>0}\operatorname{Ker}_{e^ur L^2}(D^{\pm}).$$

\section{Von Neumann algebras, foliations and index theory}
\noindent The main reference here is the original paper of Alain Connes \cite{Cos} or the book by Moore and Schochet \cite{MoSc}.
Since we shall be interested, in a future paper to formulate boundary value problems we choose to work with the equivalence relation $\mathcal{R}$ of the foliation since we believe, with Ramachandran \cite{Rama} that's the most natural ambient where to write boundary conditions. With its natural Borel structure $\mathcal{R}$ is a Borel groupoid with start, range and composition defined by:
$$s(x,y)=y,\,r(x,y)=x,\,(x,y)\cdot (y,z)=(x,z).$$ The $r$--fiber $r^{-1}(x)$ is denoted by $\mathcal{R}^x=\{(x,y):y\in L_x\}.$  
If $\{H_x\}$ is a Borel field of Hilbert spaces on the base $X$, a representation of $\mathcal{R}$ is a functor $U$ from $\mathcal{R}$ to the category of Hilbert spaces i.e. $U(x,y):H_x\longrightarrow H_y$ is a unitary isomorphism for every $(x,y)\in \mathcal{R}$. This functor has to be measurable in a precise sense \cite{Cos}. A longitudinal measure is a collection of measures $\{\nu^x\}_{x\in X}$ with $\nu^x$ supported on $\mathcal{R}^x$ that is $\mathcal{R}$--invariant i.e. $\nu^x=\nu^y$ if $(x,y)\in \mathcal{R}$. If $\nu$ is a longitudinal measure one can define a kernel in the sense of measure theory \cite{revuz} pushing forward by left traslation i.e. putting $R(\nu)_{\gamma}:=\gamma \cdot \nu^x$, $\gamma \in \mathcal{R}$, $s(\gamma)=x$. One calls $\nu$ proper if $R(\nu)$ is proper. For a proper longitudinal measure left traslation gives a representation $L^\nu$ of $\mathcal{R}$ valued on the field of $r$--fibers $H_x=L^2(\mathcal{R}^x,\nu^x).$ A representation is square integrable if it is equivalent to some subrepresentation of a denumerable union of $L^\nu$ for a proper $\nu$.

\noindent Longitudinal measures pair with transverse measures. This pairing process is well explained in \cite{MoSc}. We shall deal only with transverse measures coming from holonomy invariant transverse measures for the foliation. Remember that a Borel set $T\subset X$ is called a Borel transversal if it intersects each leaf in a atmost denumerable set. The set of all Borel transversal is a $\sigma$--ring. A finitely additive measure $\nu$ on this $\sigma$--ring is called holonomy invariant if for every Borel bijection $\phi:T_1\longrightarrow T_2$ with $(x,\phi(x))\in \mathcal{R}$ then $\nu(T_1)=\nu(T_2).$ Remember that for oriented foliations holonomy invariant transverse measures arise from foliated closed currents or from flows generated by vector fields tangent to the leaves. The Ruelle--Sullivan isomorphism is the correspondence between invariant measures and foliated currents. On foliated charts transverse measures can be disintegrated along the plaques. 

\noindent Suppose $U,V$ are two representation with values on the fields of Hilbert spaces $H,K$. A uniformly bounded Borel family of bounded operators $T_x:H_x\longrightarrow K_x$ intertwines $U$ and $V$ if 
\noindent $V_{(x,y)}\circ T_x=T_y \circ U_{(x,y)}$, $(x,y)$--a.e..

\noindent For square integrable representations on the fields of Hilbert spaces $H_i$ let $\operatorname{Hom}_{\mathcal{R}}(H_1,H_2)$ the vector space of the {intertwining} {operators}. A holonomy invariant measure gives rise to a quotient projection $\operatorname{Hom}_{\mathcal{R}}(H_1,H_2)\longrightarrow \operatorname{Hom}_{\Lambda}(H_1,H_2)$ given by identification modulo $\Lambda$-a.e. equality. Elements of $\operatorname{Hom}_{\Lambda}(H_1,H_2)$ are called {Random}  {operators} between the random Hilbert spaces $H_i$.
If $H_1=H_2=H$, then $\operatorname{Hom}_{\mathcal{R}}(H,H)=\operatorname{End}_{\mathcal{R}}(H)$ is an involutive algebra, the quotient via $\Lambda$ is the semifinite Von Neumann algebra\footnote{to be precise this is a $W^*$ algebra in fact it is not naturally represented on some Hilbert space. The choice of a longitudinal measure $\nu$ gives however a representation 
$\operatorname{End}_{\mathcal{R}}(H)\longrightarrow B(\int_X H_x d\Lambda_{\nu}(x))$ on the direct integral of the field $H_x$
}
$\operatorname{End}_{\Lambda}(H);$ its natural trace $\operatorname{tr}_{\Lambda}$ is defined using the coupling longitudinal/transverse measures \cite{MoSc}. Indeed to trace a field of operators one first looks at the field of its local traces\footnote{this is a notion introduced by Atiyah, the local trace of a positive (in general not trace class) operator $T$ on $L^2(Y,\mu)$ is the measure $A\longmapsto \operatorname{tr}_{L^2}(\chi_A T \chi_A)$
} \cite{MoSc}. This is a longitudinal measure that can be integrated against $\Lambda$ and the result is the trace.

\bigskip
\noindent For a vector bundle $E\longrightarrow X$ 
let $L^2(E)$ 
be the Borel field of Hilbert spaces on $X$ fixed by the leafwise square integrable sections 
$\{L^2(L_x,E_{|L_x})\}_{x\in X}$. There is a natural square integrable representation of 
$\re$ on 
$L^2(E)$ the one given by 
$(x,y)\longmapsto \operatorname{Id}:L^2(L_x,E)\longrightarrow L^2(L_y,E)$. 
\noindent Denote 
$\operatorname{End}_{\mathcal{R}}(E)$ the vectorspace of  all intertwining operators and $\operatorname{Hom}_{\Lambda}(E)$ the corresponding Von Neumann algebra.

\noindent Since we need unbounded operators we have to define measurability for fields of closed unbounded operators. We say that the field of closed unbounded operators $T_x$ is measurable if are measurable the fields of bounded operators $u_x$ and $(1+T_x^*T_x)^{-1}$ that determine univoquely the polar decomposition $T=u|T|.$
\bigskip

\noindent Next, we review some ingredients from Breuer theory of Fredholm operators on Von Neumann algebras, adapted to the semifinite case with some notions translated in the language of the essential $\Lambda$--spectrum, a straightforward generalization of the essential spectrum of a self--adjoint operator. The main references are \cite{Br2}, \cite{cp0} and \cite{cp}.
\noindent Remember that the set of projections $\mathcal{P}:=\{A\in \vo, A^*=A, A^2=A\}$ of a Von Neumann algebra, has the structure of a complete lattice i.e. for every family 
 $\{A_i\}_i$ of projections one can form their join
 $\operatorname{V} A_i$ and their meet 
 $\Lambda A_i$. 
 Then for a random operator $A\in \vo$ we can define its projection on the range $R(A)\in \mathcal{P}(\vo)$ and the projection on its kernel $N(A)\in \mathcal{P}(\vo)$ by \noindent $R(A):=\operatorname{V}\{P\in \mathcal{P}(\vo):PA=A\}$ and $N(A):=\Lambda\{P\in \mathcal{P}(\vo):PA=P\}$. If $A$ is the class of the measurable field of operators $A_x$, it is clear that $R(A)$ and $N(A)$ are the classes of $R(A)_x$ and $N(A)_x$.
 
\noindent Let $H_i$, $i=1,..,3$ be square integrable representations of $\mathcal{R}$; define
 $\Lambda$--finite rank random operators $B^{f}_{\Lambda}(H_1,H_2):=\{A \in \operatorname{Hom}_{\Lambda}(H_1,H_2) :\tru R(A)<\infty\}$,
$\Lambda$--compact random operators $B^{\infty}_{\Lambda}(H_1,H_2)$ as the norm closure of finite rank operators.
The $\Lambda$--Hilbert--Schmidt random operators are $B^2_{\Lambda}(H_1,H_2):=\{A \in \operatorname{Hom}_{\Lambda}(H_1,H_2): \tru(A^*A)<\infty\}$
and $\Lambda$--trace class operators $B^1_{\Lambda}(H)=B^2_{\Lambda}(H)B^2_{\Lambda}(H)^*= \{\sum_{i=1}^nS_iT_i^*:S_i,T_i\in B^2_{\Lambda}(H)\}$.
\noindent It is easy to check, as for $B(H)$, $H$ Hilbert space, that
$B^{*}_{\Lambda}(H)$ is a $*$--ideal in $\vo$ for $\ast=1,2,\infty$. An element $A\in B^*_{\Lambda}(H)$ iff $|A|\in B^*_{\Lambda}(H)$, we have
the inclusion 
$\bif \subset \buno \subset \bdue \subset \binf$ and $\buno=\{A\in \vo:\tru |A|<\infty\}.$
\noindent An important inequality we shall use (a proof in chapter $V$ of \cite{take}) is the following, take $A\in \buno$ and $C\in \operatorname{End}_{\Lambda}(H)$. We  have polar decompositions $A=U|A|$ and $C=V|C|$ then $|A|=U^*A\in \buno$, $|A|^{1/2}\in \bdue$ and \begin{equation}\label{disugess}|\tru (CA)|\leq \|C\|\tru|A|.\end{equation} 
\begin{dfn}\label{99990}
A random operator $F\in \operatorname{Hom}_{\Lambda}(E_1,E_2)$ is $\Lambda$--Fredholm (Breuer--Fredholm) if there exist $G\in \operatorname{Hom}_{\Lambda}(E_2,E_1)$ such that $FG-\operatorname{Id}\in B^{\infty}_{\Lambda}(E_2)$ and $GF-\operatorname{Id}\in B^{\infty}_{\Lambda}(E_1)$.
For a field of unbounded closed operators $T_x:H_1\longrightarrow H_2$ we say that it is $\Lambda$--Fredholm if the corresponding field $T_x:(\operatorname{Domain}(T_x),\|\cdot \|_{T_x})\longrightarrow H_2$ defined by the graph norm is $\Lambda$--Breuer--Fredholm.
\end{dfn}
\noindent It is straightforward to show that
a random operator $F\in \operatorname{Hom}_{\Lambda}(H_1,H_2)$ is $\Lambda$--Fredholm if and only if $N(F)$ is $\Lambda$--finite rank and there exist some finite rank projection $S\in \operatorname{End}_{\Lambda}(H_2)$ such that $R(\operatorname{Id}-S)\subset R(F).$
Then $\Lambda$--Fredholm operators $F$ have a finite $\Lambda$--index. In fact $\tru(N(F))<\infty$ and $\tru(1-R(F))\leq \tru(S)<\infty$ and one can define
$$\operatorname{ind}_{\Lambda}(F):=\tru(N(F))-\tru(1-R(F)).$$ 
\noindent 

\noindent To motivate the definition of the $\Lambda$--essential spectrum, as in \cite{Vai}
cosider a semifinite Von Neumann algebra $M$ with trace $\tau$, $S=S^*\in M$. One can define $\tau$--Breuer Fredholm operators exactly as in the definition  \ref{99990}.
The Borel functional calculus shows that $S$ is $\tau$--Breuer--Fredholm if and only if there exists $\epsilon>0$ such that $\tau(E(-\epsilon,\epsilon))<\infty$, where $E(\Delta)$ is the spectral projection of $S$ corresponding to a Borel set $\Delta$. Besides if $S=S^*$ is $\tau$--Breuer--Fredholm then $\operatorname{ind}_{\tau}S=0$.

\noindent So consider a measurable field $T$ of unbounded intertwining operators.  
 If $T$ is selfadjoint 
 (every $T_x$ is self--adjoint a.e.) the parametrized (measurable) spectral Theorem (cf. Theorem XIII.85 in \cite{Reed}) shows that for every bounded Borel function $f$ the family 
 $x\longmapsto f(T_x)$ is a measurable field of uniformly bounded  intertwining operators
 defining a unique random operator. In other words
$\{f(T_{x})\}_x\in \operatorname{End}_{\Lambda}(H).$
\noindent For a Borel set $U\subset \R$ let $\chi_T(U)$ be the family of spectral projections $\{\chi_{U}(T_x)\}_x$. Denote $H_T(U)$ the measurable field of Hilbert spaces corresponding to the family of the images $(H_T(U))_x=\chi_{U}(T)H_x$. The formula $$\mt(U):=\operatorname{tr}_{\Lambda}(\chi_T(U))=\operatorname{dim}_{\Lambda}(H_U(T))$$ defines a Borel measure on $\R$ such that \begin{equation}
\label{integrabo}
\int f d\mt=\operatorname{tr}_{\Lambda}(f(T)),\quad 
f:\R \longrightarrow [0,\infty) \textrm{ bounded and Borel}
.\end{equation}
We call $\mu_{\Lambda,T}$ the {$\Lambda$--spectral measure} of $T$.
Clearly this is not in general a Radon measure (i.e. finite on compact sets). In fact
due to the non--compactness of the ambient manifold a spectral projection of a relatively compact set of an (even elliptic) operator is not trace class. 
In the case of elliptic self adjoint operators with spectrum bounded by below this is the Lebesgue--Stiltijes measure associated with the spectrum distribution function relative to the $\Lambda$--trace. This is the non decreasing function $\lambda\longmapsto \operatorname{tr}_{\Lambda}\chi_{(-\ty,\lambda)}(T)$.
A good reference on this subject is the work of Kordyukov  \cite{koo}.
The proof of formula \eqref{integrabo} this fact easily follows starting from characteristic functions. Here the normality property of the trace plays a fundamental role. A detailed argument can be found in \cite{Peric}.
\noindent Next we introduce, inspired by \cite{Vai}  the \emph{main character} of this section.
\begin{dfn}
The essential $\Lambda$--spectrum of the measurable field of unbounded self--adjoint operators $T$ is
$$\spc(T):=\{\lambda \in \R: \mu_{\Lambda,T}(\lambda-\epsilon,\lambda+\epsilon)=\infty, \forall \epsilon>0\}.$$
\end{dfn}
\begin{lem}\label{compactness}
For Random operators the $\Lambda$--essential spectrum is stable under compact perturbation. If $A\in \operatorname{End}_{\Lambda}(E)$ is selfadjoint and $S=S^*\in B^{\infty}_{\Lambda}(E)$ then 
$$\operatorname{spec}_{\Lambda,e}(A+S)=\operatorname{spec}_{\Lambda,e}(A).$$ Then if $\tru$ is infinite i.e. 
$\tru(1)=\infty$ 
we have
$\operatorname{spec}_{\Lambda,e}(A)=\{0\}$
for every $A=A^*\in B^{\infty}_{\Lambda}(E).$ 
\end{lem}
\begin{proof}
Let $\lambda \in \operatorname{spec}_{\Lambda,e}(A)$, by definition $\operatorname{dim}_{\Lambda}H_A(\lambda-\epsilon,\lambda+\epsilon)=\infty$. Then consider the field of Hilbert spaces 
$$G_{\epsilon,x}:=\big{\{}t\in \chi_{(-\lambda-\epsilon,\lambda+\epsilon)}(A_x)H_x;\,\|S_xt\|< \epsilon \|t\|\big{\}}=H_{S_x}(-\epsilon,\epsilon)\cap H_{A_x}(-\lambda-\epsilon,\lambda+\epsilon).$$ 
This actually shows that $G_{\epsilon}
$ is $\Lambda$--finite dimensional in fact  
$H_{A_x}(-\lambda-\epsilon,\lambda+\epsilon)$ is 
$\Lambda$--infinite dimensional while $H_{S_x}(-\epsilon,\epsilon)$ is $\Lambda$--finite codimensional. This shows that $\lambda\in \operatorname{spec
}_{\Lambda,e}(A+S)$. The second statement is immediate.
\end{proof}
\noindent There is a spectral characterization of $\Lambda$--Fredholm random operators as expected from the definition of the $\Lambda$ essential spectrum.\begin{prop}
For a random operator $F\in \operatorname{Hom}_{\Lambda}(H_1,H_2)$ the following are equivalent
\begin{enumerate}
\item $F$ is $\Lambda$--Fredholm.
\item $0 \notin \spc(F^*F)$ and $0 \notin \spc(FF^*)$.
\item $0 \notin \spc\left(\begin{array}{cc}0 & F^* \\F & 0\end{array}\right)$
\item $N(F)$ is $\Lambda$--finite rank and there exist some finite rank projection $S\in \operatorname{End}_{\Lambda}(H_2)$ such that $R(\operatorname{Id}-S)\subset R(F).$
\end{enumerate} 
\end{prop}

\subsection{The splitting principle}
For elliptic operators, the stability of the spectrum under compact perturbations leads to an useful tool.
For every $x\in X$ and integer $k$ consider the Sobolev space $H^k(L_x,E)$ of sections of $E$, obtained by completion of 
$C^{\infty}_c(L_x,E)$ with respect to the $k$ Sobolev norm defined in terms of the longitudinal Levi Civita connection
$$\|s\|^2_{H^k(L_x;E)}:=\sum_{i=0}^k\|\nabla^k s\|_{L^2(\otimes^k T^*L_x;E)}^2.$$ This is the definition of a Borel field of Hilbert spaces with natural Borel structure given by the inclusion into 
$L^2$. In fact, by Proposition 
$4$ of Dixmier \cite{Dix} p.167 to prescribe a measure structure on a field of Hilbert spaces 
$H$ it is enough to give a countable sequence 
$\{s_j\}$ of sections with the property that for 
$x\in X$ the countable set 
$\{s_j(x)\}$ is complete orthonormal. 
In the appendix of the paper by Heitsch and Lazarov \cite{hl} is shown, making use of holonomy that a family with the property that each $s_j$ is smooth and compactly supported on each leaf can be choosen.
\begin{dfn}
\noindent Consider a field $T=\tx$ of continuous intertwining operators
\\ \noindent $T_x:\cic \longrightarrow \ci$. 

 \noindent  We say that $T$ is of order $k\in \mathbb{Z}$ if $T_x$ extends to a bounded operator 
  $H^m(L_x,E_{|Lx})\longrightarrow H^{m-k}(L_x,E_{|Lx})$ for each $m\in \mathbb Z$ and for $x$ a.e.

\noindent  We say that the $T$ is elliptic if each $T_x$ satisfies a G\"arding type inequality
$$\|s\|_{H^{m+k}_x}\leq C(L_x,m,k)[\|s\|_{H_x^m}+\|T_xs\|_{H_x^m}],$$
and the family $\{C(L_x,m,k)\}_{x\in X}$ is bounded outside a null set in $X$.
\end{dfn}
\noindent Since each leaf $L_x$ is a manifold with bounded geometry, for a family of elliptic selfadjoint intertwining operators $\tx$ every $T_x$ is essentially selfadjoint \cite{Vai} with domain $H^k(L_x;E_{|L_x})$. It makes sense again to speak of measurability of such a family.
\begin{dfn}\label{eqout}For two fields of operators $P$ and $P'$ say that $P=P'$ outside a compact $K\subset X$ if for every leaf $L_x$ and every section $s\in C^{\infty}_c(L_x\setminus K;E)$ then $Ps=P's$. This property holding $x$ a.e in $X$ with respect to the standard Lebesgue measure class.
\end{dfn}

\begin{thm}\label{1}{The splitting principle.} Let $P$ and $P'$ two Borel fields of (unbounded) selfadjoint order $1$ elliptic intertwining operators. If $P=P'$  outside a compact set $K\subset X$ then
$$\spc(P)=\spc(P').$$
\end{thm}
\begin{proof}
Let $\lambda \in \spc(P)$, for each $\epsilon>0$ put $\chi^{\lambda}_{\epsilon}:=\chi_{(\lambda-\epsilon,\lambda+\epsilon)}$ and $G_{\epsilon}:=\chi_{\epsilon}^{\lambda}(P)$, then $\tru(G_{\epsilon})=\infty$. The projection $G_{\epsilon}$ amounts to the Borel field of projections 
$\{\chi_{\epsilon}^{\lambda}(P_x)\}_{x\in X}.$ 
By elliptic regularity on each Hilbert space $G_{\epsilon,x}$ every Sobolev norm is equivalent in fact the spectral theorem and G{\aa}rding inequality show that for $s\in G_{\epsilon,x}$ and $k\in \mathbb{N}$
$$\|s\|_{H^{k+2}_x}\leq C(P_1,k+2)\{\|s\|_{L^2_x}+\|(P_1-\lambda)^ks\|_{L^2_x}\}\leq (C+\epsilon^k)\|s\|_{L^2_x}$$ where $C(P_1,k+2)$ is a constant bigger than each leafwise G{\aa}rding constant. 

\noindent Now choose two cut--off functions $\phi, \psi \in \cc$ with 
$\phi_{K}=1$ and $\psi_{|{\operatorname{supp}\phi}}=1.$ Consider the  fields of operators: 
$
{\xymatrix{B_{\phi}:L^2_x \ar[r]^{\chi_{\lambda}^{\epsilon}} &  G_{\epsilon,x} \ar[r]^{\phi} &L^2_x,
}}$ and

\noindent $
{\xymatrix{C_{\psi}:L^2_x\ar[r]^{\chi_{\lambda}^{\epsilon}}&(G_{\epsilon,x},\|\cdot\|_{L^2}) \ar[r]&(G_{\epsilon,x},\|\cdot\|_{H^k})\ar[r]^{\psi} &  H^1_x
}}$for a $k$ sufficiently big in order to have the Sobolev embedding theorem to hold.
\noindent We declare that $C^*_{\psi}C_{\psi}\in \vo$ is 
$\Lambda$--compact. In fact consider by simplicity the case in which $\psi$ is supported in a foliation chart 
$U\times T$. The integration process shows that the trace of 
$C^*_{\psi}C_{\psi}$ is given by the integration on 
$T$ of the local trace on each plaque 
$U_t=U \times \{t\}.$ Now the operator 
$C^*_{\psi,x}C_{\psi,x}$ is locally traceable by Theorem 1.10 in Moore and Schochet \cite{MoSc} since by the Sobolev embedding the range of $C_{\Psi}$ is made of continuous sections (the fact that each Sobolev norm is equivalent on $G_{\epsilon}$ makes the teorem appliable i.e don't care in forming the adjoint w.r.t. $H^1$ norm or $L^2$). These local traces are uniformly bounded in $U \times T$ from the uniformity of the G{\aa}rding constants for the family since we are multiplying by a compactly supported function $\psi$. 
Actually we have shown that $C_{\psi}^*C_{\psi}$ is $\Lambda$--trace class.
There follows from Lemma \ref{compactness} that the projection $\widetilde{G}_{\epsilon}:=\chi_{(-\epsilon^2,\epsilon^2)}(C^*_{\psi}C_{\psi})$ is $\Lambda$--infinite dimensional in fact $\spc(C^*_{\psi}C_{\psi})=\{0\}.$
\noindent Now $1-B_{\phi}$ is $\Lambda$--Fredholm  since $B_{\phi}$ is $\Lambda$--compact then its kernel has finite $\Lambda$--dimension. Also since $C^*_{\psi}C_{\psi}\chi_{\epsilon}^{\lambda}=C^*_{\psi}C_{\psi}$ then $\widetilde{{G}_{\epsilon}}\chi_{\epsilon}^{\lambda}=\widetilde{G}_{\epsilon}$ and $(1-B_{\phi})\widetilde{G}_{\epsilon}=(1-\phi)\widetilde{G}_{\epsilon}\subset \operatorname{domain}(P')$ is $\Lambda$--infinite dimensional.
\noindent Now take $s\in \widetilde{G}_{\epsilon}$, from the definition
$\|\psi s\|^2_{H^1}=\langle C_{\psi}s,C_{\psi}s\rangle_{H^1}=\langle C^*_{\psi}C_{\psi}s,s\rangle_{L^2}\leq \epsilon^2 \|s\|^2_{L^2}$ then
from the identity $(P'-\lambda)(1-\phi)s=-([P,\phi]+(1-\phi)(P-\lambda))s$ it follows that
$$\|(P'-\lambda)(1-\phi)s\|_{L^2}\leq C\|\psi s\|_{H^1}+\|(P-\lambda)s\|_{L^2}\leq \epsilon (1+C)\|s\|_{L^2}.$$
Finally the spectral theorem for (unbounded) self adjoint operators shows that $(1-\phi)\widetilde{G}_{\epsilon}\subset \chi_{(\sigma,\tau)}(P')$ with  
$\sigma=\lambda-\epsilon(1+C),\tau=\lambda+\epsilon(1+C)$. In particular $\lambda \in \spc(P')$.
\end{proof}
\begin{cor}
Consider two foliated manifolds $X$ and $Y$ with cylindrical ends or, more generally with bounded geometry with holonomy invariant measures $\Lambda_1$, $\Lambda_{2}$ and   bounded geometry vector bundles $E_1\longrightarrow X$ and $E_2\longrightarrow Y$. Suppose there exist compact sets $K_1\subset X$ and $K_2\subset Y$ such that outside $X\setminus K_1$ and $Y\setminus K_2$ are isometric with an isometry that identifies every geometric structure as the bundles and the foliation with the transverse measure. If $P$ and $P'$ are operators as in Theorem \ref{1} with $P=P'$ on $X\setminus K_1 \simeq Y\setminus K_2$ in the sense of definition \ref{eqout} then
$$\operatorname{spec}_{\Lambda_1,e}(P)=\operatorname{spec}_{\Lambda_2,e}(P').$$
\end{cor}
\begin{proof}
The proof of \ref{1} can be repeated word by word till the introduction of the element $(1-\phi)\tilde{G}_{\epsilon}$ that can be considered as an element of $\operatorname{End}_{\Lambda_2}(E_2)$ through the fixed isometry.
\end{proof}
\section{Analysis of the Dirac operator}\label{finitt}

\subsection{Finite dimensionality of the index problem}

\noindent Consider the leafwise Dirac operator $D$. Its measurability property is addressed in \cite{Rama} where is showen to be equivalent  
to that of the field of bounded operators $(D+i)^{-1}:L^2\longrightarrow H^1$ on the field of natural Sobolev spaces. 
\noindent Remember that $D$ is odd with respect to the chiral grading. The operator $D^+$ is called the longitudinal chiral Dirac operator. We shall continue to denote with $D^+$ its unique $L^2$ closure.
 In general this is not a Breuer--Fredholm operator. In fact Fredholm properties are governed by its behavior at the boundary i.e its restriction to the base of the cylinder $\partial X_0$. In the one leaf situation $D^+$ is Fredholm in the usual sense if and only if $0$ is not in the continuous spectrum of $D^-D^+$  or equivalently if the continuous spectrum has a positive lower bound \cite{Me}.
However what is still true in this case is that the $L^2$ kernels of $D^+$ and $D^-$ are finite dimensional and made of smooth sections. The difference $$\operatorname{dim}_{\Lambda}\operatorname{Ker}_{L^2}(D^+)-\operatorname{dim}_{\Lambda}\operatorname{Ker}_{L^2}(D^-)$$ is by definition the $L^2$--chiral index of $D^+$. It gives the usual Fredholm index when the operator is Fredholm. Notice that in the non Fredholm case the $L^2$ index is not stable under compactly supported perturbations.
\noindent We are going to show that in our foliation case the chiral index problem is $\Lambda$--finite dimensional in the following sense: 

\noindent $\bullet$ By the parametrized measurable spectral theorem  the projections on the $L^2$--kernels of $D^{\pm}$ belong to the Von Neumann algebras of the corresponding bundles, 
$\chi_{\{0\}}(D^+)\in \operatorname{End_{\Lambda}(E^{\pm})}$ and
decompose as a Borel family of bounded operators $\{\chi_{\{0\}}(D^{\pm})_x\}_x$ amounting to the projections on the $L^2$ kernels of $D^{\pm}_x$. Furthermore they are implemented by a Borel family of uniformly smoothing Schwartz kernels.

\noindent $\bullet$ The family of projections above give rise to a longitudinal measure on the foliation. These measure are the local traces 
$U\longmapsto \operatorname{tr}_{L^2(L_x)}[\chi_U\cdot \chi_{\{0\}}(D^{\pm})_x \cdot \chi_U]$ where for a Borel $U\subset L_x$ the operator $\chi_U$ acts on $L^2(L_x)$ by multiplication.
In terms of the smooth longitudinal Riemannian density these measures are represented by the pointwise traces of the leafwise Schwartz kernels. We prove that these local traces has the following finiteness property completely analog to the Radon property for compact foliated spaces.
\bigskip

\noindent {\bf{Finiteness property for local traces of projections on the kernel.}} 

\noindent Consider a leaf $L_x$. This is a bounded geometry manifold with a cylinder $\partial L_x \times \R^+$. We claim that for every compact $K\subset \partial L_x$ 
\begin{equation}\label{latracciaefinita} \operatorname{tr}_{L^2(L_x)}[\chi_{K\times \R^+}\cdot \chi_{\{0\}}(D^{\pm})_x \cdot \chi_{K\times \R^+}]<\ty\end{equation} Since this list of items is aimed to the definition of the index, the (rather long) proof of inequality \eqref{latracciaefinita} is postponed immediately after. We limit ourselves here to say that is the relevant form of \emph{elliptic regularity} in our situation.

\noindent $\bullet$ The integration process of a longitudinal measure against a transverse holonomy invariant measure immediately shows that the integrability condition above is sufficient to assure finite $\Lambda$--dimensionality of the $L^2$ kernels of $D^{\pm}$. Indeed it is sufficient to choose an interior transversal $S_1$ and a boundary transversal $S_2$ i.e. a transversal contained into $\partial X_0$ and apply the usual integration
displaying $X$ as measure--theoretically fibering over the union $S_1\cup S_2$. Now the integral has two terms. The second integral, on $S_2$ is finite thanks to the finiteness property above in fact the situation here is a fibered integral of a standard Radon measure on the base times a finite measure. The interior term is finite thanks to proposition 4.22 in \cite{MoSc}.

\noindent $\bullet$ Finally 
 the chiral $\Lambda$--$L^2$--index can be defined to be the real number $$\operatorname 
{Ind}_{L^2,\Lambda}(D^+):=\tru (\chi_{\{0\}}(D^+))-\tru (\chi_{\{0\}}(D^-))\in \R.$$

\noindent {\bf{Proof of finiteness property of the local trace of kernel projections}}
\begin{proof}\label{2344}
It is clear that it suffices to prove the property for each operator $(\cdot)_{|\partial_x \times \R^+}\chi_{\{0\}}(D^+_x)$. 
Let us consider the operator $D^+$ on a fixed leaf $L_x$. This is a bounded geometry manifold with a cylindrical end $\partial L_x\times \R^+=\{y\in L_x:r(y)\geq 1\}$ where
the operator can be written in the form $B+\partial/{\partial t}$
acting on sections of $F\longrightarrow \partial L_x\times \R^+$.
The boundary operator $B$ is essentially selfadjoint on $L^2(\partial L_x;F)$ on the complete manifold $\partial{L_x}$ (see \cite{Ch} and \cite{ChGrTa} for a proof of self--adjointness using finite propagation speed tecniques). 

\noindent We are going to remind the Browder--G{\aa}rding type generalized eigenfunction expansion for $B$ (see \cite{Di} 11, 300--307, \cite{DuSc} and \cite{Rama} for an application to a A.P.S foliated and Galois covering index problem). 

\noindent According to Browder--G{\aa}rding there exist
\begin{enumerate}
\item a sequence of smooth sectional maps 
$e_j:\R \times \partial L_x \longrightarrow {F}$ i.e. $e_j$ is measurable and for every $\lambda \in \R$, 
$e_j(\lambda,\cdot)$ is a smooth section of ${F}$ over $\partial L_x$ such that
$Be_j(\lambda,x)=\lambda e_j(\lambda,x).$

\item a sequence of measures $\mu_j$ on $\R$ such that the map
$V:C^{\infty}_c(\partial L_x;F)\longrightarrow \bigoplus_j L^2(\R,\mu_j)$ defined by 
$(Vs)_j(\lambda)=\langle s,e_j(\lambda,\cdot) \rangle_{L^2(\partial L_x)}$ (use the Riemannian density) extends to an Hilbert space isometry 
$V:L^2(\partial L_x; F)\longrightarrow \bigoplus_j L^2(\R,\mu_j)=:\mathcal{H}_B$ which intertwines Borel spectral functions $f(B)$ with the operator defined by multiplication by $f(\lambda)$ with domain 
$\operatorname{dom}f(B)=\Big \{ s:\sum_j \int_{\R}|f(\lambda)|^2|(Vs)_j(\lambda)|^2d\mu_j(\lambda)<\infty \Big \}.$ 
\end{enumerate}
\noindent Notice that $e_j(\lambda,\cdot)$ need not be square integrable on $L_x$. Taking tensor product with $L^2(\R)$ we have the isomorphism \begin{equation}\label{is}L^2(\partial L_x \times \R^+,F) \simeq L^2(\partial L_x,{F})\otimes L^2(\R) \xrightarrow{\sim} [\oplus_j L^2(\R,\mu_j)]\otimes L^2(\R^+)={\mathcal{H}}_B\otimes L^2(\R^+)\end{equation} where $R^+=(0,\infty)_r$. 
\noindent Under the identification $W:=V \otimes \textrm{Id}$ the operator $D^+$ is sent into $\lambda + \partial_r$ acting on ${\mathcal H}_B\otimes L^2({\R^+})$.
\noindent Now let $s$ be an $L^2$--solution of $D_xs=0$. By elliptic regularity it restricts to the cylinder as an element  \\ \noindent $s(x,r)\in C^{\infty}(\R^+,H^{\infty}(\pal;F))\cap L^2(\R^+;L^2(\pal,F))$ solving $(\partial_r+B)s=0$. Then, from a straightforward computation 
$
\partial_r (Vs)_j(\lambda,t)=-\lambda (Vs)_j(\lambda,r).$
That's to say, all the $L^2$ solutions of $D^+=0$ under the representation $V$ on the cylinder are zero $\lambda\leq 0,\, \mu_j(\lambda)$--a.e.  for every $j$. 
Decompose, for fixed $a>0$
\begin{equation}\label{decomp}L^2(\pal \times \R^+; F)=L^2(\R^+;\mathcal{H}_B([-a,a]))\oplus L^2(\R^+,\mathcal{H}_B(\R\setminus [-a,a]))\end{equation} 
where the notation is $\mathcal{H}_B(\Delta)$ for the range of the spectral projection associated to $\chi_{\Delta}$.
Let $\Pi_{\leq a}$ and $\Pi_{>a}$ respectively be the hortogonal projections corresponding to \eqref{decomp}. Let $\pkp$ be the $L^2$ projection on the kernel. There is a composition $\Pi^a:={ \Pi_{\leq a}} \circ(\cdot)_{|\pal\times \R^+}\circ \pkp$ defined through 
$
{
\xymatrix{ L^2(L_x) \ar[r] & \operatorname{Ker}_{L^2}(D_x^+) \ar[r]&L^2(\partial L_x\times \R^+) \ar[r]&L^2(\R^+;\mathcal{H}_{B}([-a,a])).
}}$ With the Browder--Garding expansion one can see that the elements $\xi$ belonging to the space $\Pi^aL^2(L_x)$ are of the form \begin{equation}\label{sha}\xi=\chi_{(0,\infty)}(\lambda)e^{-\lambda t}\zeta_0\end{equation} with $\zeta_0=\zeta_{0j} \in H^{\infty}(\pal;{F})$ to be univocally determined using boundary conditions. Formula \eqref{sha} allows to define\footnote{this is clearly inspired by Melrose definition \cite{Me} Chapter 6} the "boundary datas" mapping 
$
\textrm{BD}:\Pi^aL^2(L_x;F)\longrightarrow \mathcal{H}_B((0,a]),$

\noindent $W^{-1}(\chi_{(0,a]}(\lambda)\zeta_0 e^{-\lambda t})\longmapsto W^{-1}(\chi_{(0,a]}(\lambda)\zeta_0).
$
This is continuous and injective in fact injectivity is obvious while continuity follows at once from the simple estimate \cite{Vai}
\begin{equation}\label{74328}
\|\xi\|_{L^2(\pal\times \R^+)}
=1/(2a)\|\chi_{[-a,a]}\zeta_0 \|_{\mathcal{H}_B}.
\end{equation}
\noindent Now choose an orthonormal basis $s_m=f_m\otimes g_m \in L^2(\partial L_x\times \R^+,{F})$ and a boundary compact set $A\subset \partial L_x$, then put $\chi_{A^{\square}}=\chi_{A\times (0,\infty)}(x,r)$. Consider the operator $\chi_{A^{\square}}\Pi^{a}\chi_{A^{\square}}$ acting on $L^2(L_x;F)$. Notice that $\Pi^a$ acts on $s_m$ via the natural embedding $L^2(\partial L_x)\subset L^2(L_x)$ then
\begin{equation}\label{trac}\textrm{tr}(\chi_{A^{\square}}\Pi^{a}\chi_{A^{\square}})=
\sum_m\langle \chi_{A^{\square}}\Pi^{a}\chi_{A^{\square}} s_m,s_m\rangle_{L^2(\pal\times \R^+)}.\end{equation}
Write $\textrm{BD}[\Pi^a \chi_{A^{\square}} s_m]=W^{-1}[\chi_{(0,a]}(\lambda)\zeta_0^{(m)}]$ hence $[\Pi^a \chi_{A^{\square}} s_m]=\chi_{(0,a]}(\lambda)\zeta_0^{(m)}e^{-\lambda t}.$ 
The sequence 
$\chi_{(0,a]}\zeta_0^{(m)}$ is bounded by \eqref{74328}. Then \eqref{trac} becomes
\begin{equation}\label{conint}
\textrm{tr}(\chi_{A^{\square}}\Pi^a\chi_{A^{\square}})=
\sum_m \int_{\R^+}\int_{\R \times \mathbb{N}}\chi_{(0,a]}(\lambda)\zeta_0^{(m)}e^{-\lambda t}\overline{\Big\{{W(\chi_{A^{\square}}s_m)}\Big\}}d\mu({\lambda})dt
\end{equation} where $\mu$ is the direct sum of the $\mu_j$'s.
\noindent The Last term of equation above can be seen to be equal to\noindent $$
W(\chi_{A^{\square}}s_m)W(\chi_{A^{\square}}f_m\otimes g_m)=V(\chi_A(x)f_m(x))g_m(t)$$ but this can be estimated, using Cauchy--Schwartz and the simple trick
 $e^{-(\lambda-a)}e^{-a}=e^{-\lambda}$ by
\begin{equation*}
 \sum_m C\Big\{\int_{\R \times \mathbb{N}}\chi_{(0,a]}|V(\chi_A f_m)|^2d\mu(\lambda)dr\Big\}^{1/2}\leq C\sum_m \langle \chi_A \mathcal{H}_B((0,a])\chi_A f_m,f_m\rangle. \end{equation*}  Finally this is equal to
$C \textrm{tr}( \chi_A \mathcal{H}_B((0,a])\chi_A) <\infty,$
In fact  $\mathcal{H}_B((0,a])$ is a spectral projection of $B$ hence is uniformly smoothing by elliptic regularity.

\noindent Let us now pass to examine the operator 
$\Pi_a:= \Pi_{\geq a} { \circ(\cdot)_{|\pal\times \R^+} }\circ \pkp$ defined by the composition
\noindent $
{
\xymatrix{ L^2(L_x) \ar[r] & \operatorname{Ker}_{L^2}(D_x^+) \ar[r]&L^2(\partial L_x\times \R^+) \ar[r]&L^2(\R^+;\mathcal{H}_{B}(\R\setminus [-a,a]))
}}$ 
arising from the second addendum of the splitting \eqref{decomp}.
Let $\varphi_k$ be the characteristic function of $r\leq k$ and
$\Lambda_k:=\Pi_{\geq a}\circ \varphi_k \circ (\cdot)_{|\partial_x\times \R^+}\circ \chi_{\{0\}}(D^+_x).$ Now
\begin{align}\nonumber
\|(\Pi_a-\Lambda_k)\xi\|=
\|\Pi_{\geq a}(\varphi_k-1)(\cdot)_{|\partial L_x \times \R^+}\chi_{\{0\}}(D^+_x)\xi\|_{L^2(\partial L_x\times \R^+)}\\\label{final} 
= \int_{k}^{\infty}\int_{(a,\infty)\times \mathbb{N}}e^{-2\lambda r}|\zeta_0|^2d\mu
(\lambda)dt
 \leq e^{-2ak}\|\xi\|_{L^2(\partial L_x \times \R^+)}.\end{align}
\noindent Finally choose a compact $A\subset \partial L_x$, estimate \eqref{final} shows that $S_k:=\chi_{A^{\square}}\Lambda_k \chi_{A^{\square}}$ converges uniformly to $\chi_{A^{\square}}\Pi_a \chi_{A^{\square}}.$ 
\noindent Observe that $S_k$ is compact by Rellich theorem and regularity theory in fact $\Pi_{\textrm{Ker}(T^+)}$ is obtained by functional calculus from a rapid Borel function hence has a uniformly smoothing Schwartz--kernel. Since
$\chi_{A^{\times}}\Lambda_k \Pi_{>a}\Pi_{\textrm{Ker}(T^+)}\chi_{A^{\times}}$ is norm--limit of compact operators is compact and a compact projection is finite rank. \end{proof}

\subsection{Breuer--Fredholm perturbation}

\noindent Our main application of the splitting principle is the construction of a $\Lambda$--Breuer Fredholm perturbation of the leafwise Dirac operator. Recall the notations; $X_k:=\{r\leq k\}$, $Z_k:=\{r\geq k\}.$  
Let $\theta$ be a smooth function satisfying $\theta=\theta(r)=r$ on $Z_1$ while $\theta(r)=0$ on $X_{1/2}$, put $\dot{\theta}=d\theta/dr.$ Let $\Pi_{\epsilon}:=\chi_{I_{\epsilon}}(D^{\mathcal{F}_{\partial}})$ for $I_{\epsilon}:=(-\epsilon,0)\cup (0,\epsilon).$
\bigskip
\noindent Our perturbation will be the leafwise operator \begin{equation}\label{de2}D_{\epsilon,u}:=D+\dot{\theta} \Omega(u-\dfo \Pi_{\epsilon})\textrm{ for }\epsilon>0,\quad u\in \mathbb{R}.\end{equation} \noindent We write $D_{\epsilon,u}=D^+_{\epsilon,u}\oplus D^-_{\epsilon,u}$ and $D_{\epsilon,u,x}$ for its restriction to $L_x$, also for brevity $D_{\epsilon,0
}:=D_{\epsilon}.$

\noindent Notice that the perturbed boundary operator is \begin{equation}\label{23dice}\deuf =\dfo(1-\Pi_{\epsilon})+u=D_{\epsilon,0}^{\mathcal{F}_{\partial}}+u.\end{equation} Since for $\epsilon>0$, $0$ is an isolated point in the spectrum of $D_{\epsilon,0}^{\mathcal{F}_{\partial}}$ we see
that $\deuf$ is invertible for $0<|u|<\epsilon$. 
\noindent For further application let us compute the essential spectrum of 

\noindent $B_{\epsilon,u}=D+\Omega(u-D^{\mathcal{F}_{\partial}}\Pi_{\epsilon})$ on the foliated cylinder  $Z_0$ with product foliation $\mathcal{F}_{\partial}\times \R$. Since we deal with product structure operators we can surely think the Von Neumann algebra becomes $\operatorname{End}_{\Lambda_0}(E)\otimes B(L^2(\R))$ where $\operatorname{End}_{\Lambda_0}(E)$ is the Von Neumann algebra of the base i.e. the foliation induced on the transversal $X_0\times\{0\}$. The integration process shows that the trace is nothing but $\tru=\operatorname{tr}_{\Lambda_0}\otimes \operatorname{tr}$ where the second factor is the canonical trace on $B(L^2(\R))$.
 We can write 
\begin{align}\label{22dice}B_{\epsilon,u}^2=&\left(\begin{array}{cc}0 & -\partial_r+u+D^{\mathcal{F}_{\partial}}(1-\Pi_{\epsilon}) \\\partial_r+u+D^{\mathcal{F}_{\partial}}(1-\Pi_{\epsilon})&0\end{array}\right)^2=-\partial_r^2 \operatorname{Id}+V^2.\end{align}

\noindent Consider the spectral measure $\mu_{\Lambda_0,V^2}$ of $V^2$ on the transversal section $X_0\times \{0\}$. We claim the following facts
\begin{enumerate}
\item $\omega:=\inf \operatorname{supp}(\mu_{\Lambda_0,V^2})>0$
\item $\mu_{\Lambda,B^2_{\epsilon,u}}(a,b)=\infty,\quad 0\leq a<b,\quad \omega<b$
\item $\mu_{\Lambda,B^2_{\epsilon,u}}(a,b)=0,\quad 0\leq a<b\leq \omega.$
\end{enumerate} First of all 1. is immediately proven since \eqref{23dice} together with \eqref{22dice} implies the inclusion  $\operatorname{spec}(D_{\epsilon,u}^{\mathcal{F}_{\partial}})^2\subset [(\epsilon+u)^2,\infty).$ To prove 2. one first uses the Fourier transform in the cylindrical direction. This gives a spectral representation of 
$-\partial_r^2$ as the multiplication by 
$y^2$ on 
$L^2(\R)$. Choose some 
$\gamma<(b-\omega)/2.$ We can prove the following inclusion for the spectral projections \begin{equation}\label{specin}
\chi_{(a,\gamma+\omega)}(V^2)\otimes \chi_{(0,\gamma)}(-\partial_r^2)\subset 
\chi_{(a,b)}(B^2_{\epsilon,u}).\end{equation} In fact one can also use a (leafwise) spectral representation for 
$V$ as the multiplication operator by $x$. Then \eqref{specin} is reduced to prove the implication
$$a<x^2<\gamma+\omega,\, 0<y^2<\gamma \Rightarrow a<x^2+y^2<b.$$  From \eqref{specin} it follows 
$\mu_{\Lambda,B^2_{\epsilon,u}}(a,b)\geq \mu_{\Lambda_0,V^2}(a,\gamma+\omega)\cdot \operatorname{tr}_{B(L^2(\R))}\chi_{(0,\gamma)}(-\partial_r^2)=\infty$ in fact the first factor is non zero and the second is clearly infinite. Finally the third statement is very similar in the proof. We have shown that 
$\spc(B_{\epsilon,u}^2)=[\omega,\infty),$ then
\begin{prop}\label{312}
The operator $\deu$ is $\Lambda$--Breuer--Fredholm if $0<|u|<\epsilon$.
\end{prop}
\noindent Next we shall investigate the relations between the Breuer--Fredholm index of the perturbed operator and the $L^2$--index of the unperturbed Dirac operator. To this end we shall make use of weighted $L^2$--spaces as in the work of Melrose \cite{Me}.
\begin{dfn}
For $u\in \R$, denote 
$e^{u\theta}L^2$ the Borel field of Hilbert spaces (with obvious Borel structure given by $L^2$) $\{e^{u\theta}L^2(L_x;E)\}_x$ where, for 
$x\in X$, $e^{u\theta}L^2(L_x;E)$ is the space of distributional sections $w$ such that $e^{-u\theta}\omega\in L^2(L_x;E)$. Analog definition for weighted Sobolev spaces $e^{u\theta}H^k$ can be written.
\end{dfn}
\noindent Notice that $e^{u\theta}L^2(L_x;E)=L^2(L_x;E,e^{-2u\theta}dg_{|L_x})$ where $dg$ is the leafwise Riemannian density so these Hilbert fields correspond to the right representation of $\mathcal{R}$ with the longitudinal measure $x\in X \longmapsto e^{-2u\theta}dg_{|L_x}=r^*(e^{-2u\theta}dg)$ (transverse function, in the language of the non commutative integration theory \cite{Cos}).
\noindent The operators $D$ and its perturbation $D_{\epsilon,u}$ extend to a field of unbounded operators 
$e^{u\theta}L^2\longrightarrow e^{u\theta}L^2$ with domain 
$e^{u\theta}H^1$. Put $$e^{\infty \theta}L^2_x:=\bigcup_{\delta>0}e^{\delta \theta}L^2_x.$$
\noindent In what follows we will use, for brevity the following notation: $\partial L_x:=L_x \cap (\partial X_0\times \{0\})$ and 
$Z_x:=\partial L_x \times [0,\infty)$ for the cylindrical end of the leaf $L_x$.

\noindent For a smooth section 
$s^{\pm}$ such that 
$D_{\epsilon,u,x}^{\pm}s^{\pm}=0$ we have 
$(D_{\epsilon,u,x}^{\pm})_{|{\partial L_x\times \R^+}}(s^{\pm})_{|\partial L_x \times \R^+}=0$ that can be easily seen choosing smooth $r$--functions  
$\phi,\psi$ with $\phi_{X_0}=1$, $\psi_{Z_{1/4}}=1$, $\operatorname{supp}(\psi\subset Z_{1/8})$ and evaluating $[D_{\epsilon,u,x}^{\pm}(\phi(1-\psi)s+\phi\psi s)=0]_{|\partial L_x\times \R^+}.$

\noindent The isomorphism $W$ defined  in the proof of finiteness property for the kernel projection, can be defined also as an isomorphism  $e^{u\theta}L^2(\partial L_x \times \R^+,F) \simeq {\mathcal{H}}_B\otimes e^{u\theta} L^2(\R^+)$ in a way that solutions of $D^{\pm}_{\epsilon,u,x}s^{\pm}=0$ with conditions $s^{\pm}\in e^{\infty \theta}\cap L^2_x$ can be represented as solutions of $[\pm \partial_r+\lambda+\dot{\theta}(r)(u-\chi_{\epsilon}(\lambda)\lambda)]Ws^{\pm}=0$ with $\chi_{\epsilon}(\lambda)=\chi_{(-\epsilon,0)\cup(\epsilon,0)}(\lambda)$ acting as a multiplier on $\bigoplus_j L^2(\R,\mu_j).$ In particular (forgetting for brevity the restriction symbol)
\begin{equation}
\label{rappa}
Ws^{\pm}=\zeta_j^{\pm}(\lambda)\operatorname{exp}\{{\mp u \theta(r) \mp\lambda[r-\theta(r)\chi_{\epsilon}(\lambda)]}\}\end{equation} with suitable choosen $\zeta_j^{\pm}(\lambda)
\in L^2(\mu_j).$

\begin{prop}\label{rappacon}
Let $\epsilon>\delta>0$ and $\delta'\in \R$ then
\begin{enumerate}
\item $\xi \in \operatorname{Ker}_{e^{\delta' \theta}L^2}
(D_x^+)\Longmapsto\xi_{Z_x}=e^{-r\defox}{h}$ 
with ${h}\in \chi(\defox)_{(-\delta',\ty)}L^2_x.$
\item $\xi \in \operatorname{Ker}_{L^2}(D^+_{\epsilon,x})\Longmapsto \xi_{Z_x}=e^{-r \defox+\theta(r)\defox \Pi_{\epsilon,x}}h$, with $h\in \chi(\defox)_{(\epsilon,\ty)}L^2_x$
\item $\xi \in \operatorname{Ker}_{e^{\delta \theta}L^2}(D^+_{\epsilon,x})\Longmapsto \xi_{|Z_x}=e^{-r\defox+\theta(r)\defox \Pi_{\epsilon,x}}h,$ with $h\in \chi(\defox)_{(-\epsilon,\ty)}L^2_x,$
\end{enumerate} recall that $\Pi_{\epsilon,x}=\chi_{(-\epsilon,\epsilon)-\{0\}}(\defox).$ Moreover the following identity (as fields of operators) holds true
$$D^{\pm}e^{\mp\theta(r)D^{\mathcal{F}_{\partial}}\Pi_{\epsilon}}=e^{\mp \theta(r) D^{\mathcal{F}_{\partial}} \Pi_{\epsilon}}D_{\epsilon}^{\pm}.$$
\end{prop}
\begin{proof}

1. from the representation formula \eqref{rappa} of formal solutions for $u=0,$ $\epsilon=0$ it remains $\xi=\xi_j(\lambda)e^{-\lambda r}$. Then $e^{-\delta' \theta}\xi$ must be square integrable hence \\ \noindent $\xi_j(\lambda)=h_j(\lambda)\in \chi_{(-\delta',\ty)}(\defox)$. 
The remaining points are proved in a very similar way. The last statement is merely a computation.
\end{proof}

\noindent Solutions of $D^{\pm}_{\epsilon,x}s^{\pm}=0$ belonging to the space 
$\bigcap_{u>0}e^{u\theta}L^2(L_x;E^{\pm})$ are called 
$L^2$--\emph{extended solutions}, in symbols 
$\operatorname{Ext}(D_{\epsilon,x}^{\pm}).$ 
\begin{prop}\label{exx}
For every $x\in X$ and $0<u<\epsilon$
\begin{align}\label{primain}
&1.\,\operatorname{Ker}_{L^2}(\deup)=\operatorname{Ker}_{e^{-u\theta}L^2}(\deup)=\operatorname{Ker}_{L^2}(D^{\pm}_{\epsilon,\mp u,x})\\ &2.\,
\operatorname{Ext}(\deup)=\operatorname{Ker}_{e^{u\theta}L^2}(\deup)=\operatorname{Ker}_{L^2}(\dex).\\& 3.\, \label{terx}
\operatorname{Ker}_{L^2}(D^{\pm}_{\epsilon,x})\subset \operatorname{Ext}(D^{\pm}_{\epsilon,x})
\end{align}
\end{prop}
\begin{proof}
We show only the first equality of \eqref{primain} the others being very similar. This is a simple application of equation \eqref{rappa}. In fact, for $u=0$, $Ws^{\pm}=\zeta_j^{\pm}(\lambda)\operatorname{exp}\{{ \mp\lambda[r-\theta(r)\chi_{\epsilon}(\lambda)]}\}.$ The condition of being square integrable in $(\R,\mu_j)\otimes (\R^+,dr)$ is easily seen to be equivalent to $\zeta_j^+(\lambda)=0$ $\lambda<\epsilon$, $\lambda$--a.e and $\zeta_j^-(\lambda)=0$ $\lambda>-\epsilon$. In particular for $r\geq 1$
$Ws^{\pm}=\zeta^{\pm}_j(\lambda)e^{\mp\lambda r}\chi_{\pm \lambda \geq \epsilon}(\lambda)$ then $e^{u\theta}s^{\pm}\in L^2$ if $u<\epsilon$. For the reverse inclusion the proof is the same. For the third stament note that $e^{u\theta}L^2\subset e^{v\theta}L^2$ for every $u,v\in \R$ with $u\leq v$ then $\operatorname{Ker}_{L^2}\subset \operatorname{Ext}$.
\end{proof}
\noindent Proposition \ref{exx} shows that the mapping $x\longmapsto \ext$ gives a Borel field of closed subspaces of $L^2$. No difference in notation between the space $\operatorname{Ext}$ and $\operatorname{Ker}$ and the corresponding projection in the Von Neumann algebra will be done in the future.
Inclusion \eqref{terx} together with \ref{312} and the finiteness property of the $L^2$--kernel projection finally says that the difference 
\begin{equation}
\label{deffh}
h^{\pm}_{\Lambda,\epsilon}=\operatorname{dim}_{\Lambda}(\operatorname{Ext}(D_{\epsilon}^{\pm}))-\operatorname{dim}_{\Lambda}(\operatorname{Ker}_{L^2}(D_{\epsilon}^{\pm}))=\tru(\operatorname{Ext}(D_{\epsilon}^{\pm}))-\tru(\operatorname{Ker}_{L^2}(D_{\epsilon}^{\pm}))\in \R\end{equation} is a finite number.  
\begin{lem}\label{111}For $\epsilon>0$
\begin{enumerate}
\item $\operatorname{dim}_{\Lambda}{Ker}_{L^2}(D^{\pm}_{\epsilon})=\lim_{u\downarrow 0}\operatorname{dim}_{\Lambda}{Ker}_{L^2}(D^{\pm}_{\epsilon,\mp u})=\lim_{u\downarrow 0}\operatorname{dim}_{\Lambda}{Ker}_{L^2}(D^{\pm}_{\epsilon,\pm u})-h^{\pm}_{\Lambda,\epsilon},$
\item
$\operatorname 
{Ind}_{L^2,\Lambda}(D^+_{\epsilon})=\lim_{u\downarrow 0}\operatorname{Ind}_{\Lambda}(D^{+}_{\epsilon,u})-h^{+}_{\Lambda,\epsilon}=\lim_{u\downarrow 0}\operatorname{Ind}_{\Lambda}(D^{+}_{\epsilon,-u})+h^{-}_{\Lambda,\epsilon} $
\end{enumerate}
\end{lem}
\begin{proof}Nothing to prove here, proposition \ref{exx} says that the inclusion is constant for $u$ sufficiently small, the second one in the statement follows from the first by summation.
\end{proof}
\noindent Now define the extended solutions $\operatorname{Ext}(D^{\pm}_x)$ in the same way i.e. distributional solution of the differential operator $D^{\pm}_x:C^{\infty}_c(L_x;E^{\pm})\longrightarrow C^{\ty}_c(E^{\mp};E)$ belonging to each weighted $L^2$--space with positive weights,
$$\operatorname{Ext}(D^{\pm}_x)=\bigcap_{u>0}\operatorname{Ker}_{e^{u\theta}L^2}(D^{\pm})=\{s\in C^{-\infty}(L_x;E^{\pm});\,D ^{\pm}s=0;\, e^{-u\theta}s\in L^2\, \forall u>0\}.$$ Here we have made use of the longitudinal Riemannian density to to identify sections with sections with values densities and the Hermitian metric on $E$, in a way that one has the isomorphism 
$C^{-\infty}(L_x;E^{\pm})\simeq C^{\infty}_c(L_x;(E^{\pm})^*\otimes \Omega(L_x))^*$ to simplify the notation.

\noindent It is clear by standard elliptic regularity that the extended solutions of $D^{\pm}$ are smooth on each leaf. In fact $D^{\pm}$ a first order differential elliptic operator and one can construct a parametrix i.e. an inverse of $D^{\pm}$ modulo a smoothing operator. An operator that sends each Sobolev space onto each other (of the new, weighted metric of corse). 

\noindent By definition $\operatorname{Ext}(D^{\pm})\subset e^{u\theta}L^2$ for every $u>0$, define $\operatorname{dim^{(u)}_{\Lambda}}(\operatorname{Ext})$ as the trace in $\operatorname{End}_{\Lambda}(e^{u\theta}L^2)$ of the projection on the closure of $\operatorname{Ext}$, now we must check that under the natural inclusion $e^{u\theta}L^2\subset e^{u'\theta}L^2$ $(u<u')$ these dimensions are preserved. This is done at once in fact the inclusion $\operatorname{Ext}(D^{\pm})\subset e^{u\theta}L^2 \hookrightarrow   \operatorname{Ext}(D^{\pm})\subset e^{u'\theta}L^2$ is bounded and extends to a bounded mapping 
$\overline{\operatorname{Ext}(D^{\pm})}^{ e^{u\theta}L^2} \longrightarrow   \overline{\operatorname{Ext}(D^{\pm})}^{ e^{u'\theta}L^2}$ with dense range. Now the unitary part of its polar decomposition is an unitary isomorphism then the $\Lambda$ dimensions are preserved by the essential property of formal dimension stating that if the space of homomorphisms of two random Hilbert spaces cointaines an invertible element then the dimensions are the same \cite{Cos}.

\begin{dfn}\label{170}
The $\Lambda$--dimension of the space of extended solution is $$\operatorname{dim}_{\Lambda}\operatorname{Ext}(D^{\pm}):=\operatorname{dim}_{\Lambda} \overline{\operatorname{Ext}(D^{\pm})}^{e^{u\theta}L^2}$$ for some $u>0$.
\end{dfn}
\begin{prop}\label{hj}
\begin{enumerate}
\item $\lim_{\epsilon \downarrow 0} \operatorname{dim}_{\Lambda}\operatorname{Ker}_{L^2}(D_{\epsilon}^{\pm})=\operatorname{dim}_{\Lambda}\operatorname{Ker}_{L^2}(D^{\pm})$
\item $\lim_{\epsilon\downarrow 0} \operatorname{Ind}_{L^2,\Lambda}D^+_{\epsilon}=\operatorname{Ind}_{L^2,\Lambda}D^+$
\item $\lim_{\epsilon \downarrow 0} \operatorname{dim}_{\Lambda}\operatorname{Ext}(D_{\epsilon}^{\pm})=\operatorname{dim}_{\Lambda}\operatorname{Ext}(D^{\pm})$
\end{enumerate}
\end{prop}
\begin{proof}
1.
 let $\xi \in \operatorname{Ker}_{L^2}(D^{+}_{\epsilon,x})$ thanks to Proposition \ref{rappacon} we can represent \\ \noindent $\xi_{|Z_x}=e^{-r \defox+\theta(r)\defox \Pi-{\epsilon,x}}h,\,\, h\in \chi_{(\epsilon,\ty)}(\defox).$  From $\Pi_{\epsilon,x}h=0$ we get
 
\noindent $
D^{+}_x\xi_{|Z_x}=(D^+_{\epsilon,x}+\theta(r)D^{\mathcal{F}_{\partial}}_x\Pi_{\epsilon,x})\xi_{|Z_x}=\theta(r) \defox \Pi_{\epsilon,x}(e^{-r\defox+\theta(r)\defox \Pi_{\epsilon,x}}h)=0
$ meaning that $\operatorname{Ker}_{L^2}(D^+_{\epsilon,x})\subset \operatorname{Ker}_{L^2}(D^+).$
Moreover 

\noindent $
D^{+}_{\epsilon}(\operatorname{Ker}_{L^2}(D^+))= 
\dotto \defox \Pi_{\epsilon,x}(\operatorname{Ker}_{L^2}(D^+)\subset -\dotto \defox e^{-r \defox} \chi_{(-\epsilon,\epsilon)}(\defox)(L^2(\partial L_x \otimes L^2(\R^+)). 
$ Note that clearly $\operatorname{dim}_{\Lambda}\Big{[} \dotto \defox e^{-r \defox} \chi_{(-\epsilon,\epsilon)}(\defox)(L^2(\partial L_x \otimes L^2(\R^+))\Big{]}\longrightarrow_{\epsilon \rightarrow 0 } 0$ by the normality of the trace.
 Then the family of operators ${D^{+}_{\epsilon}}_{|\operatorname{Ker}_{L^2}(D^+)}:\operatorname{Ker}_{L^2}(D^+)\longrightarrow L^2$ has kernel  $\operatorname{Ker}_{L^2}(D^+_{\epsilon,x})$ and range with $\Lambda$ dimension going to zero. 1. follows by looking at an orthogonal decomposition 
 $\operatorname{Ker}_{L^2}(D^+)=\operatorname{Ker}_{L^2}(D^+_{\epsilon})\oplus \operatorname{Ker}_{L^2}(D^+)/\operatorname{Ker}_{L^2}(D_{\epsilon}^+).$

\noindent 2. follows immediately from $1.$

\noindent 3. consider the following commutative diagram
$$\xymatrix{\operatorname{Ker}_{e^{\delta\theta}L^2}(D^+)\ar[r]\ar[dr]^{\Psi^{+}_{\epsilon}} &\operatorname{Ker}_{e^{(\delta+\epsilon)\theta}L^2}(D^+)\\ 
&  \operatorname{Ker}_{e^{\delta\theta}L^2}(D^+_{\epsilon})\ar[u]^{\Psi^{-}_{\epsilon}}        }$$ where 
$\Psi^{\pm}_{\epsilon}=e^{\pm\theta \Pi_{\epsilon}D^{\mathcal{F}_{\partial}}}$. It is easily seen  thanks to the representation of solutions in proposition \ref{rappacon} that each arrow is injective and bounded with respect to the inclusions 

\noindent $
e^{\delta\theta}L^2\hookrightarrow e^{(\delta+2\epsilon)\theta}L^2
\hookrightarrow  e^{(\delta+\epsilon)\theta}L^2.$ Then adding to the right of the above diagram the column 
$e^{(\delta+2\varepsilon)\theta}L^2 \longrightarrow e^{(\delta+\varepsilon)\theta}L^2$ one gets a new diagram. This last column 
can be used to measure dimensions. The inequality
 $$\operatorname{dim}_{\Lambda}\operatorname{Ker}_{e^{\delta\theta}L^2}(D^+)\leq \operatorname{dim}_{\Lambda}\operatorname{Ker}_{e^{\delta\theta}L^2}(D^+_{\epsilon})\leq
\operatorname{dim}_{\Lambda}\operatorname{Ker}_{e^{(\delta+\epsilon)\theta}L^2}(D^+)$$ follows and prove immediately $3.$.\end{proof}

\section{Cylindrical finite propagation speed and \\
Cheeger--Gromov--Taylor type estimates.}
\subsection{The standard case}
\noindent A very important property of the Dirac operator on a manifold of bounded geometry $X$ is finite propagation speed for the associated wave equation.
Let $P\in \udif^1(X,E)$ uniformly elliptic first order (formally) self--adjoint operator.
 The diffusion speed of $P$ in $x$ is the norm of the principal symbol $\sup_{v\in S^*_x}|\sigma_{\textrm{pr}}(P)(x)|$ ($S_x^*$ is the fibre of cosphere bundle at $x$). Taking the supremum on $x$ in $M$ one gets the \emph{maximal diffusion speed} $c=c(P)$.

\noindent We say that an operator has \emph{finite propagation speed} if its maximal diffusion speed is finite.

\noindent Generalized Dirac operator associated to bounded geometry datas (manifold and Clifford structure) has finite propagation speed in fact its principal symbol is Clifford multiplication.

\noindent Now an application of the spectral theorem shows that for every initial data $\xi_0 \in C^{\infty}_c(X,E)$ there is a unique solution $t\mapsto \xi(t)$ of the Cauchy problem for the wave equation associated with $P$
\begin{equation}\label{waveproblem} 
\left\{\begin{array}{l} {\partial \xi}/\partial t-iP\xi=0, \\
\xi(0)=\xi_0,
\end{array}
\right.
\end{equation}
This solution is given by the application of the one parameter group of unitaries $\xi(t)=e^{itP}\xi_0$. By the Stone theorem the domain of $P$ is invariant under each unitary $e^{itP}$ and $e^{itP}$is bounded from each Sobolev space $H^s$ into itself. In particular the domain of $P$ is invariant under each unitary $e^{itP}$.
\begin{lem}
For $\theta$ suitably small and $x\in M$, $\|\xi(t)\|_{L^{2}B(x,\theta-ct)}$ is decreasing in $t$. In particular 
$\textrm{supp}(\xi_0)\subset B(x,r) \Longmapsto \textrm{supp}(e^{itP}\xi_0)\subset B(x,r+ct). $
\end{lem}
\noindent The proof is in J. Roe's book \cite{Roel} Prop. 5.5 and lemma 5.1.. First one proves that for a small geodesic ball of radius $r$ the function $\|e^{itP}\xi_0\|_{L^2(B(x,r-ct))}$ is decreasing. This is called the \underline{energy estimate}; then the second step follows easily.

\noindent For operators with finite propagation speed one has the representation formula in terms of the inverse Fourier transform 
\begin{equation}\label{fourier rep}f(P)=\int_{\R}\hat{f}(t)e^{itP}{dt}/{2\pi}.\end{equation} The integral converges in the weak operator topology, namely
$\langle f(P)x,y\rangle=\int  \hat{f}(t)\langle e^{itP}x,y\rangle {dt}/{2\pi},$ for every $x,y\in L^2(X;E)$. If $X=S^1$ this is just Poisson summation formula.

\noindent Now formula \eqref{fourier rep} leads us to an easy method to obtain pointwise extimates of the Schwartz kernel $[f(P)]$ for a class Schwartz function $f$. In fact due to the ellipticity of $P$, $f(P)$ is a uniformly smoothing operator and $[f(P)] \in UC^{\infty}(X\times X;\operatorname{End}(E))$.
\begin{prop}
Take some section $\xi \in L^2(X;E)$ supported into a geodesic ball $B(x,r)$
then the following estimate holds true
\begin{equation}\label{base}
\|f(P)\xi\|_{L^2(X- B(x,R))}\leq (2\pi)^{-1/2}\|\xi\|_{L^2(X)}\int_{\R-I_R}|\hat{f}(s)|ds,\end{equation}
where 

\noindent $I_R:=(-\frac{r-R}{c},\frac{r-R}{c})$ with the convention that $I_R=\emptyset$ if $R\leq r.$
\end{prop}
\begin{proof}See \cite{Vai}
\end{proof}
\noindent So the point of view is the following;

\noindent 1. Mapping properties of $f(D)$ will lead to pointwise estimates on the Schwartz kernel of $f(D)$ \cite{ChGrTa}. More precisely; start with a compactly supported section $s$, suppose we can extimate the $L^2$ norm of the image 
$f(D)s$ on a small ball $B$ at some distance $d$ from the support, then by elliptic regularity (G\"arding inequality) and Sobolev embeddings we can extimate the kernel $[f(D)]$ pointwisely.

\noindent 2. This $L^2$ norm, $\|f(D)s\|_{L^2(B)}$ is extimated in terms of the $L^1$ norm of the Fourier transform $\|\hat{f}\|_{L^1(\R)}$. As $d$ increases we can cut large and large intervals around zero in $\R$. This means that the relevant norm becomes $\|\hat{f}\|_{L^1(\R-I_d)}$ where $I_d$ is an interval containing zero.
 The limit case of this phenomenon says that spectral functions made by functions with compactly supported Fourier transforms will produce \underline{properly supported operators} i.e. operators whose kernel lies within a $\delta$--neighborhood of the diagonal. For an application of finite propagation speed in Foliations one can look at the paper \cite{Roeff} where is showen that spectral functions $f(D)$ where $f$ has compactly supported Fourier transform belong to the $C^*$--algebra of the foliation.
\noindent Estimate \eqref{base} is the starting point. In the following proposition pointwise estimates on the Schwartz kernel are worked out from this mapping properties. This is a very rough version of the ideas contained in \cite{ChGrTa}. A complete proof in \cite{Vai}.
\begin{prop}\label{est}
Let $r_1>0$ sufficiently small, $x,y\in X$ put $$R(x,y):=\max\{0,d(x,y)-r_1\}$$ and $\bar{n}:=[n/2+1]$, $n=\operatorname{dim}X$, $I(x,y):=(-R(x,y)/c,R(x,y)/c)$. For a Schwartz class function $f\in \mathcal{S}(\R)$
\begin{equation}\label{schw}\Big{|}\nabla_x^l\nabla_y^k[f(P)]_{(x,y)}\Big{|}\leq \mathcal{C}(P,l,k,r_1)\sum_{j=0}^{2\bar{n}+l+k}\int_{\R-I(x,y)}\Big{|}\hat{f}^{(j)}(s)\Big{|}ds.\end{equation}
\end{prop}
\noindent In the special case of the \underline{heat} \underline{kernel} $[f(P)]=[e^{-tP^2}]$ when $f(x)=e^{-tx^2}$, $\hat{f}(s)=(2t)^{-1/2}e^{-s^2/4t}$,
one can use the Hermite polynomials to write the derivatives of $\hat{f}$. this leads to the well known estimates 
\begin{align}
\label{heats}
|\nablal \nablak [P^me^{-tP^2}]_{(x,y)}| \leq 
\Bigg{\{}{\begin{array}{c}\mathcal{C}(k,l,m,P)t^{-m/2}e^{-R^2/6c^2t},\,\,\,t>T \\   
\mathcal{C}(k,l,m,P)e^{R^2/6c^2t},\,\,\,d(x,y)>2r_1
\end{array}t\in \R^{+}.}
\end{align}
There's also a relative version of Proposition \ref{est} in which two \underline{differential}, formally self--adjoint uniformly elliptic operators $P_1$ and $P_2$ are considered. More precisely relative means that $P_1$ acts on 
$E_1\longrightarrow X_1$ and $P_2$ acts on 
$E_2\longrightarrow X_2$ with open sets $U_1 \subset X_1$, $U_2\subset X_2$ and isometries 
$\varphi:U_1\longrightarrow U_2$ and 
$\Phi:{E_1}_{|U_1}\longrightarrow  {E_2}_{|U_2}$ with 
$\Phi \circ \varphi=\varphi \circ \Phi$
making possible to identify $P_1$ with $P_2$ upon $U=U_1=U_2$ i.e.
$\Phi( P_1 s)=P_2( \Phi s),\quad s\in C^{\infty}_c(U_1;E_1)$ where $\Phi$ is again used to denote the induced mapping on sections

\noindent  $\Phi:C^{\ty}_c(U_1;E_1)\longrightarrow C^{\ty}_c(U_2;E_2),\,(\Phi s)(y):=\Phi_{\varphi^{-1}(y)}s(\varphi^{-1}(y)).$ Thanks to the identification one calls $P=P_1=P_2$ over $U$.
\noindent Then the relative version of the estimate \eqref{schw} is contained in the following proposition.
\begin{prop}\label{asd}
Choose $r_2>0$ and let \underline{$x,y$ be in $U$}. Set $J(x,y):=\Big{(}\dfrac{-Q(x,y)}{c},\dfrac{Q(x,y)}{c}\Big{)}$ where
$Q(x,y):=\max\{\min \{d(x,\partial U);d(y,\partial U) \}-r_2;0\}.$ 
For a class Schwartz function $f\in \mathcal{S}(\R)$,
$$|\nablal \nablak ([f(P_1)]-[f(P_2)])_{(x,y)}|\leq \mathcal{C}(P_1,k,l,r_2)\sum_{j=0}^{2\bar{n}+l+k}\int_{\R-J(x,y)}|\hat{f}^{(j)}(s)|ds. $$ More precisely the reason of the dependence of the constant only to $P_1$ is that it depends upon ${P_1}_{|U}$ where the operators coincide.
\end{prop}
\begin{prop}\label{rel12}
\noindent The relative version of \eqref{heats} is
\begin{align*}
|\nablal \nablak ([P_1^me^{-tP_1^2}]-[P_2^m &e^{-tP_2^2}])_{(x,y)}|  \nonumber \leq \Bigg{\{}{\begin{array}{c}\mathcal{C}(k,l,m,P_1)t^{-m/2}e^{-Q(x,y)^2/6c^2t},\,t>T \\   
\mathcal{C}(k,l,m,P_1)e^{-Q(x,y)^2/6c^2t},\,t>0
\end{array}}
\end{align*}
for $x,y\in U$, $d(x,\partial U)$, $d(y,\partial U)>r_2$. 
\end{prop}
\subsection{The cylindrical case}
In this section our manifold $L$ will be the generic leaf of the foliation i.e.  start with a manifold with bounded geometry $L_0$ with boundary $\partial L_0$ composed of possibly infinite connected components and a product type Riemannian metric near the boundary. Glue an infinite cylinder $Z_0=\partial L_0 \times [0,\infty)$ with product metric and denote $L:= L_0 \cup_{\partial L_0} Z_0$. Let $E\longrightarrow L$ be an Hermitian Clifford bundle. Every notation of section \ref{geom} is keeped on. Recall that $E_{|Z_0}=F\oplus F$.
\begin{dfn}
We say that a first order uniformly elliptic (formally) selfadjoint operator $T\in \operatorname{Op}^1(L;E)$ has product structure if
\begin{enumerate}
\item $T$ restricts to $L_0$ and $Z_0$ i.e. $\operatorname{supp}(T s)\subset L_0 (Z_0)$ if $s$ is supported on $L_0\,(Z_0).$
\item $T_{|L_0}$ is a uniformly elliptic \underline{differential} operator.
\item $T$ restricts to the cylinder to have the form 
$$T_{|Z_0}=c(\partial_r)\partial_r+\Omega B(r)=\left(\begin{array}{cc}0 & B(r)-\partial_r  \\B(r)+\partial(r) & 0 \end{array}\right)$$
 for a smooth mapping $B:\R^+\longrightarrow \operatorname{Op}^1(\partial L_0;E)$ with values on the subspace of uniformly elliptic and selfadjoint operators. Furthermore suppose that $B(r)\cong B$ is constant for $r\geq 2$.
\end{enumerate} 
\end{dfn}
\noindent  However this is only a model embracing our Breuer--Fredholm perturbation of the Dirac operator in fact 
\begin{equation}\label{asc}(D_{\epsilon,u,x})_{|\partial_x \times \R^+}=c(\partial_r)\partial_r+\Omega\underbrace{(\dot{\theta}u-\dot{\theta}D^{\mathcal{F}_{\partial}}\Pi_{\epsilon}+D^{\mathcal{F}_{\partial}})}_{B(r)}.\end{equation} In this sense every result from here to the end of the section has to be thought applied to $D_{\epsilon,u}.$

\noindent Some words about the smoothness condition on the mapping $B$. Here we shall make use only of pseudodifferential operators with uniformly bounded symbols, (almost everywhere they will be smoothing operators) hence the smoothness condition of the family is the usual one. In particular this is the smoothness of the family of operators acting on the fibers of $\partial L_0\times \R^+\longrightarrow \R^+$, $B(t)\in \operatorname{Op}^1(\partial L_0\times \{t\};E)$. If $U$ is a coordinate set for $\partial L_0$ such a family is determined by a smooth mapping $p:\R^+\longrightarrow \operatorname{S}_{\operatorname{hom}}^1(U)$ in the space of polihomogeneous symbols. Here smooth means that each derivative $t\longmapsto {d^k \sigma}/{dt^k}$ is continuous as a mapping with values in the space of symbols (with the symbols topology, see \cite{Va})
\noindent Again the spectral theorem shows that for a compactly supported section $\xi_0\in C^{\ty}_c(L;E)$ there is a unique solution $t\mapsto \xi(t)$ of the Cauchy problem \eqref{waveproblem} for the wave equation associated with $T$. This solution is
given by the application of the wave one parameter group $e^{itT}$ 
with the same properties written above in the standard case.
\begin{prop}\label{59}{\bf{ Cylindrical finite propagation speed.}} Let $U=\partial L_0\times (a,b)$ $0<a<b$ and $B(U,l)=\{x\in L:d(x,U)<l\}$. For $\xi_0\in C^{\ty}_c(L;E)$ let $\xi(t)=e^{itT}\xi_0$ the solution of the wave equation. If $\alpha<a$ the function $\|\xi(t)\|_{L^2(B(U,\alpha-t))}$ is not increasing in $t$. In particular $\operatorname{supp}(\xi_0)\subset U \Longmapsto \operatorname{supp}(\xi(t))\subset B(U,t).$
\end{prop}
\begin{proof}
The product structure of the operator makes us possible to repeat the standard proof of the energy estimates and finite propagation speed. The proof consist in showing that
$
d/dt \|\xi(t)\|^2_{L^2(B(U,\alpha-t))}\leq 0$. Everything works because $T$ has product structure, the integration domain is a product and the operator $B(t)$ is selfadjoint on the base. Notice also that that $\xi(t)_{|\partial L_0 \times \{r\}}$ is in the domain of $B(r)$ by the theorem of Stone \cite{Reed} (however it is certainly true for operators in the form of our perturbation \eqref{asc}).
\end{proof}
\noindent As a notation for a subset $H\subset L$ and $t\geq0$ put $H\ast t:=B(H,t)\cup \partial L_0\times (\alpha-t,\beta+t)$ where 
$\alpha:=\inf\{r(z):z\in H\cap Z_0\}$ and 
$\beta:=\max\{r(z):z\in H\cap Z_0\}$ in other words $H\ast t$ is the set of points at distance $t$ from $H$ in the cylindrical direction.

\noindent It is clear from \eqref{59} the inclusion $$\operatorname{supp}(e^{itT}\xi)\subset \operatorname{supp}(\xi)\ast |t|.$$ Then the cylindrical basic Cheeger--Gromov--Taylor estimate similar to \eqref{base} is obtained noting that proposition \ref{59} is certainly true if the propagation speed is $c$, for a section $\xi$ supported into a ball $B(x,r_0)$ and $f\in \mathcal{S}(\R)$ let
$I_R:=(-(R-r_0)/c,(R-r_0)/c)$ if $R>r_0$ and $I_R=\emptyset$ if $r\leq R$ then, 
\begin{align*}
\|f(P)\xi\|_{L^2(L-B(x,r_0)\ast R)}&=\Big{\|}(2\pi)^{-1/2}\int_{\R}\hat{f}(s)e^{isP}\xi ds\Big{\|}_{L^2(L-B(x,r_0)\ast R)} 
 \\ &\leq (2\pi)^{-1/2}\|\xi\|_{L^2(L)}\|\hat{f}\|_{L^1(\R-I_{R})},
\end{align*} since $\operatorname{supp}{e^{isP}\xi}\cap (L-B\ast R)=\emptyset$ for $|t|<(R-r_0)/c$.
\begin{prop}Choose two points on the cylinder $z_1=(x_1,s_1)$ and $z_2=(x_2,s_2)$ with $s_i>r_1$, $|s_1-s_2|>2r_1$, put $I(z_1,z_2):=\Big{(}-\dfrac{|s_1-s_2|+r_1}{c},\dfrac{|s_1-s_2|-r_1}{c}\Big{)}$ then for $f\in \mathcal{S}(\R)$,
$$|\nabla_{z_1}^l\nabla_{z_2}^k[f(P)]_{(z_1.z_2)}|\leq \mathcal{C}(P,l,k)\sum_{j=0}^{2\bar{n}+l+k}\int_{R-I(z_1,z_2)}|\hat{f}(s)^{(j)}|ds$$ with $\bar{n}:=[n/2+1]$
\end{prop}
\begin{proof}The proof is identical to the proof of Proposition 3.9 in \cite{Vai}. There is only a subtle point we need to reckle, it is when one let $P^j$ act on $[f(P)]_{(x,\bullet)}$. This is perfectly granted by the smoothing properties of $f(P)$ in fact, let the bundle be $L\times \R$ and identify distributions with functions through the Riemannian density.
The operator $f(P)$ extends to an operator from compactly supported distributions to distributions (actually takes values on smooth functions).
 Consider the family of Dirac masses $\delta_y(\cdot)$ concentrated at $y$,
first note that $[f(P)]_{(x,y)}=(f(P)\delta_y(\cdot))(x)$ in fact this is, by  selfadjointness, equivalent to $$\langle f(P)\delta_y,s\rangle=\langle \delta_y,f(P)s\rangle=\int [f(P)]_{(z,y)}t(z)dz.$$ Now the Sobolev embedding theorem says that $\delta_y\in H^k(X)$ with $k<-n/2$ with norms uniformly bounded in $y$. Since $f(P)$ maps every Sobolev space into each other Sobolev space, every section $[f(P)]_{(x,\bullet)}$ (and the symmetric one by selfadjointness) is in the domain of $P^j$.

\end{proof}
\begin{cor}
With the notations of the proposition above
\begin{enumerate}
\item If $|s_1-s_2|>2r_1$, $s_i>r_1$
\begin{equation}\label{stil}|\nabla_{z_1}^l\nabla_{z_2}^k [P^me^{-tP^2}]_{(z_1,z_2)}|\leq \mathcal{C}(k,l,m,P)e^{-\dfrac{(|s_1-s_2|-r_1)^2}{6t}}\end{equation}
\item Let $\psi_1,\psi_2$ compactly supported with supports at $r$--distance $d$ on the cylinder, then for the operator norm \begin{equation}\label{1333}
\|\psi_1 P^me^{-tP^2}\psi_2\|\leq \mathcal{C}(m,\psi_1,\psi_2)e^{-d^2/6t}\,,t>0.
\end{equation}
\item The relative version of \eqref{stil} is 
\begin{equation}\label{asdl}
|\nabla_{z_1}^l\nabla_{z_2}^k [P^me^{-tP^2}-T^me^{-tT^2}]_{(z_1,z_2)}|\leq \mathcal{C}(k,l,m,P)e^{\{{-(\min\{s_1,s_2\}-r_2)^2/6t}\}}.\end{equation}
\end{enumerate}
\end{cor}
\begin{proof}
The second statement follows immediately from the first one while the third can be proven exactly as proposition \ref{asd}.
\end{proof}
\section{The eta invariant}\label{eta}
\subsection{The classical eta invariant} The eta invariant of Atiyah Patodi and Singer appears for the first time in the following theorem that we write in the cylindrical case.
\begin{thm}
\noindent Let $X$ a compact manifold with boundary $Y$ and product type metric on a collar $Y\times [0,1]$, attach an infinite cylinder $Y\times [-\infty,0]$ to get the elonged manifold $\hat{X}$:=X. Let $D:C^{\infty}(X;E)\longrightarrow C^{\infty}(X;F)$ a first order differential elliptic operator with product structure near the boundary i.e.
$D=\sigma(\partial_u+A)$ where $\sigma E_{|Y}\longrightarrow F_{|Y}E$ is a bundle isomorphism, $\partial_u$ is the normal interior coordinate and $A$ is the boundary self--adjoint elliptic operator. Then the operator $D$ extends to sections of the bundles extended to $\hat{X}$ and has a finite $L^2$ index given by the formula $$\operatorname{ind}(D)=\operatorname{dim}_{L^2(\hat{X},E)}(D)-\operatorname{dim}_{L^2(\hat{X},E)}(D^*)=\int_X\alpha_0(x)dx-\eta(0)/2-\dfrac{h_{\ty}(E)-h_{\ty}(F)}{2}.$$
\noindent Here the defect number $\eta(0)$, is called the \underline{spectral asymmetry} or the \underline{eta invariant} of $A$ and is obtained as follows:

\noindent the summation on the non negative eigenvalues of $A$, $\eta(s):=\sum_{\lambda\neq 0}\operatorname{sign}(\lambda)|\lambda|^{-s}$ converges absolutely for $\operatorname{Re}(s)>>0$ and extends to a meromorphic function on the whole $s$--plane with regular value at $s=0$. Moreover if the asymptotic expansion has no negative powers of $t$ then $\eta(s)$ is holomorphic for $\operatorname{Re}(s)>-1/2$. That's the case of Dirac type operators on Riemannian manifolds.
\end{thm}
\subsection{The foliation case}
\noindent The existence of the eta invariant for the leafwise Dirac operator on a closed foliated manifold was shown by Peric \cite{Peric} and Ramachandran \cite{Rama}. In fact they build different invariants, Peric works with the holonomy groupoid of the foliation and Ramachandran with the equivalence relation but the methods are essentially the same. So consider a compact manifold $Y$ with a foliation and a longitudinal Dirac structure i.e. every geometrical structure needed to form a longitudinal Dirac--type operator acting on the tangentially smooth sections of the bundle $S$, $D:C^{\ty}_{\tau}(Y;S)\longrightarrow (Y;S)$. In our index formula $Y$ will be a transverse section of the cylinder sufficiently far from the compact piece and $D$ is the operator at infinity. Suppose also that a transverse holonomy invariant measure $\Lambda$ is fixed.

\noindent The passage from the summation $\eta(s)=\sum_{\lambda}\operatorname{sign}(\lambda)|\lambda|^{-s}$ which deals with the discrete spectrum to a continuous spectrum and family version  is given by the definition of Euler gamma function 
$\operatorname{sign}(\lambda)|\lambda|^{-s}=\dfrac{1}{\Gamma(\frac{s+1}{2})}\int_0^{\ty}t^{\frac{s-1}{2}}\lambda e^{-t\lambda^2}dt.$ Each bounded Borel spectral function of $D$ belongs to the Von Neumann algebra of the foliation arising from the regular representation of the equivalence relation on the Borel field of $L^2(S)$. Replace the summation by integration w.r.t. the spectral measure of $D$  and (formally) change the integration to define the eta function of $D$ as
\begin{align}\label{etafun}
\eta_{\Lambda}(D;s):=\int_{-\infty}^{\infty}\operatorname{sign}(\lambda)|\lambda|^{-s}d\mu_{D}(\lambda)=\dfrac{1}{\Gamma(\frac{s+1}{2})}\int_0^{\ty}t^{\frac{s-1}{2}}\operatorname{tr}_{\Lambda}(De^{-tD^2})dt.
\end{align}
\noindent We shall use also the notation $$\eta_{\Lambda}(D;s)_k:=\int_k^{\ty}t^{\frac{s-1}{2}}\operatorname{tr}_{\Lambda}(De^{-tD^2})dt,\,\,\,\eta_{\Lambda}(D;s)^k:=\int_0^{k}t^{\frac{s-1}{2}}\operatorname{tr}_{\Lambda}(De^{-tD^2})dt$$

\begin{thm}(Ramachandran)
The eta function \eqref{etafun} is a well defined meromorphic function for $\operatorname{Re}(s)\leq 0$ with eventually simple poles at $(\operatorname{dim}{\mathcal{F}}-k)/2$, $k=0,1,2,....$. It is regular at $0$ and the value $\eta_{\Lambda}(D;0)$ is called the foliated eta invariant of $D$.
\end{thm}
\begin{proof}
We give a sketch of the proof since we shall use it in our computations for the eta invariant of the perturbed Dirac operator.

\noindent {\bf{First step.}} Prove that for every $s\in \mathbb{C}$ with $\Re(s)\leq 0$ the integral $\int_1^{\infty}t^{\frac{s-1}{2}}\operatorname{tr}_{\Lambda}(De^{-tD^2})dt$ is convergent; This is  proven by simple estimates with the use of the spectral measure. In particular here the spectral measure $\mu_{\Lambda,D}$ is tempered i.e. there exists some positive $l$ such that $\int \dfrac{1}{(1+|x|^l)}d\mu_{\Lambda,D}<\ty$. In fact this measure corresponds to 
a positive functional  \cite{Rama} $I:\mathcal{S}(\R)\longrightarrow \R,\, I(f)=\operatorname{tr}_{\Lambda}(f(D)).$ 
The same is obviously true for the square $D^2=|D|^2.$

\noindent {\bf{Second step.}} The examination of the finite piece $\int_0^{1}t^{\frac{s-1}{2}}\operatorname{tr}_{\Lambda}(De^{-tD^2})dt$
is done using the expansion of the Schwartz kernel of the leafwise operator $De^{-tD^2}$. One can prove that there exists a family of tangentially smooth and locally computable functions $\{\Psi_m\}_{m\geq 0}$ \footnote{in the case of the holonomy groupoid the $\Psi_m$ are locally bounded i.e. bounded on every set in the form of $r^{-1}K$ for $K$ compact in $Y$} so that the kernel $K_{t}(x,y,n)$ ($n$ the transverse parameter) of the leafwise bounded operator $De^{-tD^2}$ has the asymptotic expansion
\begin{equation}\label{scwe}K_{t}(x,x,n)\sim \sum_{m\geq 0}t^{(m-\operatorname{dim}\mathcal{F}-1)/2}\Psi_m(x,n).
\end{equation} 
Moreover $\Psi_m=0$ for $m$ even. The proof is an adaptation of the classical situation \cite{Cos}. Now, thanks to the expansion \eqref{scwe}, since the operator $De^{-tD^2}$ is $\Lambda$ trace class and the trace is the integral of the Schwartz kernel against the transverse measure we get the corresponding expansion for the trace
\begin{equation}\label{23}\dfrac{1}{\Gamma(\frac{s+1}{2})}\int_0^1 t^{\frac{s-1}{2}}\operatorname
{tr}_{\Lambda}(De^{-tD^2})dt\sim \sum_{m\geq 0}\dfrac{2}{s+m-\operatorname{dim}{\mathcal{F}}}\int_Y \Psi_m d\lambda\end{equation}
where $\int \Psi_m d\lambda=\Lambda(\Psi_m dg)$ is the effect of the integration of the tangential measures 
$x\longmapsto {\Psi_m}_{|l_x}\times dg_{|l_x}$. From \eqref{23} we see that the eta function has a meromorphic continuation to the whole plane with (at most) simple poles at $(\operatorname{dim}{\mathcal{F}}-k)/2, \,\,k=0,1,2,....$

\noindent {\bf{Third step.}} Regularity at the origin. \noindent If $p=\operatorname{dim}{\mathcal F}$ is even we have said that the coefficients $\Psi_m$ of the development \eqref{scwe} are zero for $m$ even, then the eta function is regular at $0$. If $p$ is odd the regularity at zero follows from a very deep result of Bismut and Freed \cite{Bifree}. In fact they showed that the ordinary Dirac operator satisfies a remarkable cancellation property,
 $$\operatorname{tr}(De^{-tD^2})=O(t^{1/2}).$$ Since the $\Lambda$--trace can be, as pointed out by Connes \cite{Cos}, locally approximated by the regular trace their result applies to our setting to give
 $K_t(x,x,n)\sim \sum_{m\geq p+2}\underbrace{t^{(m-p-1)/2}\Psi(x,n)}_{\textrm{almost everywhere}},$ and the regularity at the origin follows immediately.
\end{proof}

\subsection{Eta invariant for perturbations of the Dirac operator}
\noindent Let us consider slightly more general operators

 {\bf{1.}} \underline{{\bf{$P=D+K$}}} where $K\in \operatorname{Op^{-\ty}}$ is leafwise uniformly smoothing obtained by functional calculus, $K=f(D)$ where $f$ is a bounded Borel function supported in  $(-a,a)$.

\noindent Start with the computation
\begin{align}\label{pa}
Qe^{-tQ^2}-De^{-tD^2}&=
 Ke^{-t(D+K)^2}-D\int_0^1 e^{-s(D+K)^2}(KD+DK+K^2)e^{(t-s)D^2}ds.
\end{align} The family \eqref{pa} converges to $0$ as $t\rightarrow 0$ in the Frechet topology of kernels in $\operatorname{Op}^{-\ty}$ with uniform transverse control. Indeed for kernels $K(x,y,n)$ ($n$ is the transverse parameter) one uses foliated charts to define seminorms that involve derivatives w.r.t. $x,y$. From \eqref{pa} one gets the development
\begin{equation}\label{asdfrw}\operatorname{tr}_{\Lambda}(Qe^{-tQ^2})\sim_{t\rightarrow 0}
\sum_{m=0}t^{\frac{m-\operatorname{dim}\mathcal{F}-1}{2}}\int_Y \Psi_jd\Lambda +\operatorname{tr}_{\Lambda}(K)+g(t)\end{equation} where $g\in C[0,\ty)$ with $g(0)=0$. Then an asymptotic development for $\eta_{\Lambda}(Q)(0)_1$ as \eqref{23} follows. 

{\bf{2.}} \underline{The smooth family $u\longmapsto Q_u:=D+K+u$}. 

\noindent The function $\operatorname{tr}_{\Lambda}(Q_ue^{-tQ_u^2})$ is smooth ( same identical proof as \cite{Vai}) then, since $Q_u^{'}=\operatorname{I},$ 

\noindent $
\partial_u \operatorname{tr}_{\Lambda}(Q_u e^{-tQ_u^2})
=(1+2t\partial_t)\operatorname{tr}_{\Lambda}(Q^{'}_u e^{-tQ_u^2}).
$  By integration 
\begin{align}\label{322}
\partial_u \eta_{\Lambda}(Q_u,s)_1=
 \int_0^1 \dfrac{t^{(s-1)/2}}{\Gamma{(\frac{s+1}{2})}}\operatorname{tr}_{\Lambda}(Q_u^{'}e^{-Q_u^2})-\dfrac{s}{\Gamma(\frac{s+1}{2})}\int_0^1 t^{(s-1)/2}\operatorname{tr}_{\Lambda}(Q_u^{'}e^{-tQ_u^2})dt.
\end{align}\noindent Now proceed as before using the asymptotic development of the heat kernel for $D+u$ \footnote{$(D+u)^2$ is a generalized Laplacian},
$\operatorname{tr}_{\Lambda}(Q_u^{'}e^{-tQ_u^2})=\operatorname{tr}_{\Lambda}(Q_u^{'}e^{-tQ_u^2})\sim \sum_{m\geq 0}a_m(D+u)t^{(m-\operatorname{dim}\mathcal{F})/2}+g(t)$ where $g\in C[0,\ty)$, $g(0)=0.$
\noindent We see that the integral in \eqref{322} admits a meromorphic expansion around zero in $\mathbb{C}$ with zero as a pole of almost first order. Then the derivative $\partial_u \eta_{\Lambda}(Q_u,s)_1$ is holomorphic around zero. The identity
$\partial_u \operatorname{Res}_{|s=0} \eta_{\Lambda}(Q_u,s)_1=
 \operatorname{Res}_{|s=0}\partial_u \eta_{\Lambda}(Q_u,s)_1=0$ says that $\operatorname{Res}_{|s=0} \eta_{\Lambda}(Q_u,s)_1$ is constant in $u$ then the function $\eta_{\Lambda}(Q_u,s)_1$ is holomorphic at zero since
 $\eta_{\Lambda}(Q_0,s)_1$ is holomorphic in $0$.

 {\bf{3.}} \underline{Families in the form $Q_u=D+u+\Pi D$ for a spectral projection $\Pi=\chi_{(-a,a)}(D)$}.
\begin{prop}\label{0348}The eta invariant for $Q_u$ exists and satisfies
 $$\eta_{\Lambda}(Q_u)=\operatorname{LIM}_{\delta \rightarrow 0}\int_{\delta}^{1}\dfrac{t^{-1/2}}{\gamma(1/2)}\operatorname{tr}_{\Lambda}(Q_ue^{-tQ_u^2})dt+\int_{1}^{\ty}\dfrac{t^{-1/2}}{\gamma(1/2)}\operatorname{tr}_{\Lambda}(Q_ue^{-tQ_u^2})dt$$ where $\operatorname{LIM}$ is the constant term in the asimptotic development in powers of $\delta$ as $t\rightarrow 0$. Moreover for every $u\in \R$ and $a>0$,
 
 \noindent a. $\eta_{\Lambda}(Q_u)-\eta_{\Lambda}(Q_0)=\operatorname{sign}(u)\operatorname{tr}_{\Lambda}(\Pi)$
 
 \noindent b. $\eta_{\Lambda}(Q_0)=1/2\eta_{\Lambda}(Q_u)+1/2\eta_{\Lambda}(Q_{-u})$
 
 \noindent c. $|\eta_{\Lambda}(D)-\eta_{\Lambda}(Q_0)|=|\eta_{\Lambda}(\Pi D)|\leq \mu_{\Lambda,D}((-a,a))$.
 \end{prop}

\begin{proof}The first statement can be proved as above.

\noindent 
a. using the spectral measure we have to compute the difference  
\begin{align}\nonumber \int_0^{\ty}t^{-1/2}\int_{\R}(x+u-\chi x)e^{-t(x+u-\chi x)^2} -(x-\chi x)e^{-t(x-\chi x)^2}d\mu_{\Lambda,D}(x)\dfrac{dt}{\Gamma(1/2)}\end{align} where $\chi=\chi_{(-a,a)}(x).$ Split the integral on $\R$ into two pieces, $|x|> a$ and $|x|\leq a$.

\noindent First case, $|x|> a.$ Changing the integration order the first integral is 

\noindent $\Gamma(1/2)^{-1}\int_{|x|>a}\int_0^{\ty}(x+u)t^{-1/2}e^{-t(x+u)^2}
dt d\mu_{\Lambda,D}(x)$ and performing the change of variable $\sigma:=t(x+u)^2$ in the second we see that the difference is zero.
 Second case $|x|<a$, the second integral is zero, the first is
\begin{align}&\nonumber
\int_0^{\ty}t^{-1/2}\int_{-a}^{a}ue^{-tu^2}d\mu_{\Lambda,D}(x)\dfrac{dt}{\Gamma(1/2)}=\underbrace{\noindent \int_0^{\ty}u|u|\sigma^{-1/2}e^{-\sigma}\dfrac{d\sigma}{|u|^2}}_{tu^2=\sigma}\dfrac{\operatorname{tr}_{\Lambda}(\Pi)}{\Gamma(1/2)}=\operatorname{sign}(u)\operatorname{tr}_{\Lambda}(\Pi).
\end{align}b. and c. easily follow from a.
\end{proof} 

\section{The index formula}\label{qww}This section is devoted to the proof of the index formula. Computations that are not different from that of \cite{Vai} are omitted for brevity.
First we introduce the \underline{supertrace} notation. Since the bundle $E=E^+\oplus E^-$ is $\mathbb{Z}_2$--graded, there is a canonical  Random operator $\tau$ obtained by passing to the $\Lambda$--class of the family of involutions $\tau_x:=\left(\begin{array}{cc}\operatorname{Id}_{L^2(L_x;E^+)} & 0 \\0 & -\operatorname{Id}_{L^2(L_x;E^-)}\end{array}\right).$
Now the \underline{$\Lambda$--{supertrace}} of $B\in \operatorname{End}_{\Lambda}(E)$ is by definition
 $\operatorname{str}_{\Lambda}(B):=\operatorname{tr}_{\Lambda}(\tau B).$

\bigskip

\noindent Now according to proposition \ref{312} for  $\eppu$ the perturbed operator $\deu$ is $\Lambda$--Breuer--Fredholm. 
Consider the heat operator $\hdeupx$ on the leaf $L_x$. This is a uniformly smoothing operator with a Schwartz kernel (remember that the metric trivializes densities and $[\bullet]$ means Schwartz kernel) in the space $C^\ty$--sections that are bounded together with each covariant derivative, $[\hdeupx]\in UC^{\infty}(L_x\times L_x;\operatorname{End}(E)).$ It is a well know fact the convergence for $t\longrightarrow \infty$ in the Frechet space of $UC^{\infty}$ sections to the kernel of the smoothing projection on the $L^2$--null space,
$$\lim_{t\rightarrow \infty}[\hdeupx]=[\chi_{\{0\}}(\deux)].$$ This is  a consequence of the continuity of the functional calculus from rapid decaying Borel functions into $C^\ty$ uniformly bounded sections, $RB(\R)\longrightarrow UC^{\infty}(\operatorname{End(E)})$ applied to the sequence of functions $e^{-t\lambda^{2}}\longrightarrow \chi_{\{0\}}$ in $RB(\R).$
 Choose cut--off functions $\phi_k\in C_c^{\infty}(X)$ such that ${\phi_k}_{|X_k}=1$, ${\phi_k}_{|Z_{k+1}}=0.$ 
 \noindent The measurable family of bounded operators $\{\pk \hdeupx \pk\}_{x\in X}$ gives an intertwining operator $\pk \hdeut \pk\in \operatorname{End}_{\mathcal{R}}(L^2(E))$ hence a random operator 
 
 \noindent $\pk \hdeut \pk\in \operatorname{End}_{\Lambda}(L^2(E)).$

\begin{lem}
The random operator $\pk \hdeut \pk\in \operatorname{End}_{\Lambda}(L^2(E))$ is $\Lambda$--trace class. The following formula (iterated limit) holds true 
\begin{equation}\label{limita}\indu(\deupp)=\stru(\chi_{\{0\}}(\deu))=
\lim_{k\rightarrow \infty}\lim_{t\rightarrow \infty}\stru(\pk \hdeut \pk).
\end{equation}
\end{lem}
\begin{proof}For the first statement there's nothing to proof, it is essentially the closed foliated manifold case. The local traces define  tangential measures that are $C^{\infty}$ in the direction of the leaves  while Borel and uniformely bounded (by the uniform ellipticity of the operator) in the transverse direction and we are integrating against the transverse measure on a compact set. More precisely we are evaluating the mass of a compact set through the measure $\Lambda_{h}$ where $h$ is the longitudinal measure that on the leaf $L_x$ is given by
$A\longmapsto \int_{A}\operatorname{str}_{\operatorname{End}(E)}[e^{-tD^2_{\epsilon,u}}]_{\operatorname{diag}}dg_{|L_x},$  with $\operatorname{str}_{\operatorname{End}(E)}$ the pointwise supertrace defined on the space of sections of $\operatorname{End}(E)\rightarrow X$ by $(\operatorname{str}_{\operatorname{End}(E)}\gamma)(x):=\operatorname{tr}_{\operatorname{end}(E_x)}(\tau(x)\gamma(x)).$
\noindent The limit formula \eqref{limita} is nothing but the Lebesgue dominated convergence theorem applied two times, first
$\stru(\chi_{\{0\}}(D_{\epsilon,u}))=\lim_{k\rightarrow \infty}
\stru(\pk \chi_{\{0\}}(D_{\epsilon,u})\pk)$ but for fixed $k$ one finds $$\stru(\pk \chi_{\{0\}}(D_{\epsilon,u}))\pk)=\lim_{t\rightarrow \infty}\stru (\pk \hdeut \pk).$$
The possibility to apply the dominated convergence theorem is given again by the integration process in fact as written above every tangential measure has smooth density w.r.t to the Riemannian metric and convergence is within the Frechet topology of $C^{\infty}$ functions. 

\end{proof}\noindent Now, Duhamel formula $d/dt \stru(\pk \hdeut \pk)=-\stru(\pk D^2_{\epsilon,u}\hdeut \pk)$ integrated between $s$ and $\infty$ leads to the identity
$$\lim_{t\rightarrow \infty}\stru(\pk \hdeut \pk)=
\stru(\pk \hdeups \pk)-\int_s^{\infty}
\stru(\pk D^2_{\epsilon,u}\hdeut\pk)dt.$$ Note that the right--hand side is independent from $s>0$. Then 
\begin{equation}\label{indint}\indu(\deupp)=\lim_{k\rightarrow \infty}
\left[
\stru (\pk \hdeups \pk)-\int_s^{\infty}\stru(\pk D^{2}_{\epsilon,u}\hdeut\pk) dt \right].\end{equation}
Split the integral into 
$$\int_s^{\infty}\stru(\pk D^{2}_{\epsilon,u}\hdeut\pk) dt =
\int_s^{\sqrt{k}}\stru(\pk D^{2}_{\epsilon,u}
\hdeut\pk) dt +\int_{\sqrt{k}}^{\infty}
\stru(\pk D^{2}_{\epsilon,u}\hdeut\pk) dt$$
\noindent  and make the following definitions
$$\xymatrix{{\alpha}_0(k,s)=\stru(\pk \hdeups \pk), &  \beta_0(k,s)=\int_s^{\infty} 
\stru(\pk D^2_{\epsilon,u}\hdeut\pk)dt\\
\beta_{01}(k,s)=\int_s^{\sqrt k}\stru(\pk D^{2}_{\epsilon,u}
\hdeut\pk) dt,& 
\beta_{02}(k,s)=\int_{\sqrt{k}}^{\infty}\stru(\pk D^{2}_{\epsilon,u}
\hdeut\pk) dt
}$$

\noindent Then $\beta_0(k,s)=\beta_{01}(k,s)+\beta_{02}(k,s)$ and 
\begin{equation}
\label{indf}
\indu(\deupp)=\lim_{k\rightarrow \ty}[\alpha_0(k,s)-\beta_0(k,s)]=[\alpha_0(k,s)-\beta_{01}(k,s)-\beta_{02}(k,s)].\end{equation}
\noindent Let us start with $\beta_{01}$.
\begin{lem}\label{309}Let $\eta_{\Lambda}(\deuf)$ be the Ramachandran eta--invariant for the perturbed operator $\deuf$ on the foliation at the infinity. Then the following limit formula is true
$$\lim_{k\rightarrow \ty}\operatorname{LIM}_{s\rightarrow 0}\beta_{01}(k,s)=\lim_{k\rightarrow \infty}\operatorname{LIM}_{s\rightarrow 0}\int_s^{\sqrt k}\stru(\pk D^{2}_{\epsilon,u}e^{-t\deuq}\pk)ds,=1/2 \eta_{\Lambda}({\deuf})$$\noindent where as usual $\operatorname{LIM}_{s\rightarrow 0}g(s)$ is the constant term in the expansion of $g(s)$ in powers of $s$ near zero.
\end{lem}
\begin{proof}
The integrand is
 $
\stru(\pk \deuss \pk)=-1/2\stru(c(\partial_r)\partial_r(\pkd)\deus).
$
\noindent In the next we shall use the notation $[a,b]:=ab-(-1)^{|a|\cdot |b|}ba$ for the Lie--superbracket\footnote{everything we say about super--algebras can be found in \cite{BeGeVe}} on the Lie--superalgebra of $\mathbb{C}$--linear endomorphisms of 
$L^2(X,E^+\oplus E^-)$ while, when the standard bracket is needed we write $[a,b]_{\circ}:=ab-ba.$ Notice that
$[\alpha,ab]=[\alpha,a]b+(-1)^{|\alpha|\cdot|a|}a[\alpha,b].$
\noindent Remember the definition of 
$\deu$, in the cylinder. It can be written $$\deu=D+\dot{\theta}\Omega(u-\deuf)=c(\partial_r)\partial_r+Q$$ with the Clifford multiplication $c(\partial_r)=\left(\begin{array}{cc}0 & -1 \\1 & 0\end{array}\right)$ and $Q$ is $\R^+$--invariant in fact it acts on the transverse section.
 The next identities are also useful
$$\deu=\left(\begin{array}{cc}0 & \deum \\\deupp & 0\end{array}\right),\quad \hdeut=\left(\begin{array}{cc}\eum & 0 \\0 & \eup\end{array}\right),$$
$$\deum \eup=\eum \deum,\quad \deupp \eum=\eup \deupp.$$
These are nothing but a rephrasing of the identity $\deus=\hdeut \deu$ granted by the spectral theorem. 
\noindent Now it's time to use the Cheeger--Gromov--Taylor relative estimates. Consider the leafwise operator 
\begin{equation}\label{sconfronta}S_{\epsilon,u}:=c(\partial_r)\partial_r+\Omega(u-\deuf)\end{equation} on the infinite foliated cylinder (in both directions) $Y=\partial X_0\times \R$ with the product foliation $\mathcal{F}_{\partial}\times \R$.
\noindent Choose some point $z_0=(x_0,r)$ on the cylinder. Estimate \eqref{asdl} says that we can compare the two kernels at the diagonal leaf by leaf for large $r$ and this estimate is uniform on the leaves,
\begin{equation}\label{cgt} \|[\deussxo]-[\dessxo]\|_{(z,z)}\leq Ce^{-(r-r_2)^2/(6t)}\end{equation}
for $\underline{z=(x,r)\in L_{z_0}}$. From \eqref{cgt}, since the derivatives of $\pk$ are supported on the cylindrical portion $Z_k^{k+1}=\partial X_0\times [k,k+1]$,
\begin{align}\label{opla}
\sk |\stru&(\clib  \deus)-\stru(\clib \esm)|dt=
\int_s^{\sqrt{k}}\int_{Z_k^{k+1}}\Theta(z,t)d\Lambda_g dt
\end{align} 
where $\Lambda_g$ is the coupling of $\Lambda$ with the tangential  Riemannian measure and $\Theta(z,r)$ is the function
$\Theta(z,r):=\|c(\partial_r)\partial_r \phi^2_k[D_{\epsilon,u,z}e^{-tD^2_{\epsilon,u,z}}-S_{\epsilon,u,z}e^{-tS^2_{\epsilon,u,z}}]\|_{(z,z)}.$
Let $\mathcal{T}_k$ be a transversal of the foliation $\mathcal{F}_k$ induced on the slice $\{r=k\}$ then $\mathcal{T}_k$ is also a transversal for $\mathcal{F}$  (the boundary foliation has the same codimension of $\mathcal{F}$). The transverse measure $\Lambda$ defines also a transverse measure on the boundary foliation. Then the foliation $\mathcal{F}_{|Z_k^{k+1}}$ is fibering on $\mathcal{T}_k$ as in the diagram  $\partial{\mathcal{F}}\times[k,k+1] \longrightarrow \mathcal{T}_k$. Use this fibration to disintegrate the measure $\Lambda_{g}$. This is splitted into $d\Lambda_{\partial}\times dr$ where $\Lambda_{\partial}$ is the measure obtained applying the integration process of $\Lambda$ (restricted to $\mathcal{F}_k$ ) to the $g_{|\partial}$. In local coordinates $(r,x_1,...,x_{2p-1})\times (x_{2p},...,x_n)$ the transversal is decomposed into pieces $\mathcal{T}_k=\{(k,x_1^0,...,x_{2p-1}^0)\}\times \{(x_{2p},...,x_n)\}$ and we are taking integrals 
\begin{align}\label{itegralesplitt}\int_{\mathcal{T}_k\times\{x_1,...,x_{2p-1}\}}\int_{[k,k+1]} \Theta(r,x_1,...,x_{2p-1},x_{2p},...,x_n)dr \underbrace{dx_1\cdot\cdot \cdot dx_{2p-1}  d\Lambda(x_{2p},..,x_n)}_{\textrm{this is }d\Lambda_{\partial}}\\
\nonumber =:\int_{\mathcal{F}_k}\int_{[k,k+1]}\Theta(x,r)d\Lambda_{\partial}dr. \end{align}
\noindent Equation \eqref{itegralesplitt} can be taken as a definition of a notation that will be used next. Notice that $\int_{\mathcal{F}_k}$ contains a slight abuse of notation, in fact to follow rigorously the integration recipe 
one should write $\int_{\partial X_0\times \{k\}}$. We prefer the first to stress the fact that we are splitting 
w.r.t the foliation induced on the transversal.
With this notation in mind \cite{Vai} the right hand side of \eqref{opla} is less than 

\begin{align*}
\sk \int_{\mathcal{F}_z}\int_{[k,k+1]}\|\clib [\deus-\esm]\|_{((x,r),(x,r))}dr &d\Lambda_{\partial} dt
\\
&\leq C(e^{-k^{3/2}/c_1}+e^{-c_2/s})
\end{align*}
for sufficiently small\footnote{$\,\,y^s e^{-ay^2} \leq (\dfrac{s}{2ae})^{s/2}$ for $s,u,y,a>0$} $s$  and large $k$. This estimate says that 
 $$\lim_{k\rightarrow +\ty}\operatorname{LIM}_{s\rightarrow 0}\beta_{01}(k,s)=\lim_{k\rightarrow +\ty}\operatorname{LIM}_{s\rightarrow 0}\int_s^{\sqrt k}\stru(\clib \esm)dt.$$

\noindent Now the second integral (on the cylinder) is explicitly computable
  in fact the Schwartz kernel of the operator $\dessxo$ on the diagonal is easily checked to be

\noindent $$
\big[\dessxo\big]_{(z,z)}=
\dfrac{1}{\sqrt{4\pi t}}\Omega\big[\deufo e^{-t\deufo}\big]_{(x,x)},\,\,\,z=(x,r).
$$ In particular it does not depend on the cylindrical coordinate $r$.
\noindent Now the pointwise supertrace on $\operatorname{End}(E)$ is related to the trace on the positive boundary eigenbundle $F$ via the identity (the proof in \cite{BeGeVe})
$\operatorname{str}^E(c(\partial_r)\Omega\bullet)=-2\operatorname{tr}^F(\bullet),$ then
\begin{align*}
\int_{s}^{\sqrt{k}}\stru (c(\partial_r)\partial_r \pkd& S_{\epsilon,u}e^{-tS^2_{\epsilon,u}})dt=
\int_s^{\sqrt{k}}\int_{\mathcal{F}_0}\dfrac{1}{\sqrt{\pi t}}\operatorname{tr}^F[D^{\mathcal{F}_{\partial}}_{\epsilon,u,x}e^{-t (D^{\mathcal{F}_{\partial}}_{\epsilon,u,x} )^2   }]_{(x,x)}\cdot d\Lambda_{\partial} dt, \end{align*}
with the same argument on the splitting of measures as above.
\noindent Finally it is clear from our discussion on the $\eta$--invariant 
 (exactly proposition \ref{0348}) that

\noindent $
\lim_{k\rightarrow \ty}
\operatorname{LIM}_{s\rightarrow 0}\beta_{01}(k,s)= 
\lim_{k\rightarrow \ty}\operatorname{LIM}_{s\rightarrow 0}
\int_s^{\sqrt{k}}\int_{\mathcal{F}_0}\dfrac{1}{\sqrt{\pi t}}\operatorname{tr}^F[D^{\mathcal{F}_{\partial}}_{\epsilon,u,x}e^{-t (D^{\mathcal{F}_{\partial}}_{\epsilon,u,x} )^2   }]_{(x,x)}\cdot d\Lambda_{\partial} dt
\\=1/2 \eta_{\Lambda}(D^{\mathcal{F}_{\partial}}_{\epsilon,u}).
$
\end{proof}
\begin{lem}Since $\deu$ is $\Lambda$--Breuer--Fredholm for $\eppu$ then
$$\lim_{k\rightarrow \ty}\beta_{02}(k,s)=\lim_{k\rightarrow \ty}\int^{\ty}_{\sqrt{k}}\stru(\pk D^{2}_{\epsilon,u}\hdeut \pk)dt=0.$$
\end{lem}
\begin{proof}From the very definition of the $\Lambda$--essential spectrum there exists some positive $\sigma=\sigma(u)$ such that the projection $\Pi_{\sigma}=\chi_{[-\sigma,\sigma]}(\deu)$ has finite $\Lambda$--trace. Then, from the trick $e^{-D^2_{\epsilon,u}}=e^{-D^2_{\epsilon,u}/2}(1-\Pi_{\sigma}+\Pi_{\sigma})e^{-D^2_{\epsilon,u}/2}$
$$
|\beta_{02}(k,s)|\leq \underbrace{{\int_{\sqrt{k}}^{\ty}e^{-(t-1)\sigma
}|\stru(\pk \deuq e^{-\deuq}\pk)|dt}}_{\beta_{021}(k,s)}+\underbrace{\int_{\sqrt{k}}^{\ty}|\stru(\deuq e^{-t\deuq}\Pi_{\sigma})|dt}_{\beta_{022}(k,s)}
$$
 Now the Schwartz kernel of $(\deuq e^{-\deuq})_x$ is uniformly bounded in $x$ and varies in a Borel fashion transversally. When  forming the  $\Lambda$--supertrace we are integrating a longitudinal measure with $C^{\ty}$--density w.r.t. the longitudinal measure given by the Riemannian density. Let as usual be $\Lambda_g$ the measure given by the integration of the Riemannian longitudinal measure with the transverse measure $\Lambda$. If $A$ is a uniform bound on the leafwise Schwartz kernels of $(\deuq e^{-\deuq})$, and $\mathcal{T}_0$ is a complete transversal contained in the normal section of the cylinder (the same of Lemma \ref{309}), we can extimate 
 \\ \noindent $\beta_{021}(k,s)\leq \int_{\sqrt{k}}^{\ty}A(\Lambda_g(X_0)+\Lambda(\mathcal{\mathcal{T}_0})k)e^{-(t-1)\sigma}dt\longrightarrow_{k\rightarrow \ty}0. $
 
 \noindent For the second addendum, changing the order of integration $dt \rightarrow d\mu_{\Lambda,D_{\epsilon,u}}$,
\begin{align*}\beta_{022}(k,s)\leq C\tsi e^{-\sqrt{k}x^2}\dmd \leq C \mun([-\sigma,\sigma])\longrightarrow_{k\rightarrow \ty} 0 
\end{align*} since the $\Lambda$--essential spectrum of $D_{\epsilon,u}$ has a gap around zero and the normality property of the trace.
\end{proof}
\noindent It is time to update equation \eqref{indf},
\begin{equation}\label{passagg}
\nonumber \indu(\deupp)=\lim_{k\rightarrow \ty}[\alpha_0(k,s)-\beta_0(k,s)]=\lim_{k\rightarrow \ty}\operatorname{LIM}_{s\rightarrow 0}\alpha_0(k,s)-1/2\eta_{\Lambda}(D_{\epsilon,u}^{\mathcal{F}_{\partial}}).
\end{equation}
\begin{lem}There exists a function $g(u)$ with $\lim_{u\rightarrow 0}g(u)=0$ such that for $0<\epsilon<u$,
\begin{align*}\lim_{k\rightarrow \ty}\operatorname{LIM}_{s\rightarrow 0}\alpha_0(k,s)=\lim_{k\rightarrow \ty}\operatorname{LIM}_{s\rightarrow 0}\stru(\pk \hdeups \pk)=\langle \widehat{A}(X)\operatorname{Ch(E/S)},C_{\Lambda}\rangle +g(u).\end{align*}
\noindent Here the leafwise characteristic form $\widehat{A}(X)\operatorname{Ch(E/S)}$ is supported on $X_0$, in particular it belongs to the domain of the Ruelle--Sullivan current $C_{\Lambda}$ associated to the transverse measure $\Lambda.$\end{lem}
\begin{proof}This is the investigation of the behavior of the local supertrace of the family of the leafwise heat kernels $\operatorname{str}^E[e^{-s D^2_{\epsilon,u}}]_{|\operatorname{diag}}$ on the leafwise diagonals. We can do it dividing into three separate cases 

 {\bf{1.}} For $z\in X_0$ everything goes as in the classical computation by Atiyah Bott and Patodi \cite{AtBo} 
$\operatorname{LIM}_{s\rightarrow 0}\operatorname{str}^E[e^{-sD^2_{\epsilon,u,z}}]_{(x,x)}dg_z=\widehat{A}(X,\nabla)\operatorname{Ch}(E/S,\nabla)(x),$ where $dg_z$ is the Riemannian density on the leaf $L_z$.

 {\bf{2.}} In the middle, $z\in \partial X_0 \times [0,4]$ there's the cause of the presence of the defect function $g(u)$, more precisely we show that the asymptotic development of the local supertrace is the same for the comparison operator $S_{0,u}$ defined above 
$$\operatorname{str}^E([e^{-sD^2_{\epsilon,u,z}}])_{(z,z)}\simeq \sum_{j\in \mathbb{N}}a_j(S_{0,u})_{(z)}s^{(j-\operatorname{dim}{\mathcal{F}})/2}$$
 with coefficients 
$a_j(S_{0,u})$ smoothly depending on $u$ satisfying 
$a_j(S_{0,u})=0$ for $j\leq \operatorname{dim}\mathcal{F}/2$ 

 {\bf{3.}} Away from the base of the cylinder $z=(y,r)\in Z$ $r>4$ we find
$[e^{-D^2_{\epsilon,u,z}}]_{(y,r)}=0.$
Below the proofs of these facts.

\bigskip

 {\bf{1.}} We can consider the doubled manifold $2X_0$ so that we can apply the relative estimate of type Cheeger--Gromov--Taylor in the non--cylindrical case (the perturbation starts from the cylinder).
Proposition \ref{rel12}
 shows that the two Schwartz kernels of the Dirac operator and the perturbed operator $D_{\epsilon,u}$ have the same development as $t\rightarrow 0$ 
The local computation of Atiyah Bott and Patodi, or the Getzler rescaling (\cite{Me},\cite{Ge}) can be performed as in the classical situation.

 {\bf{2.}} We are going to use an argument of comparison with the leafwise operator

\noindent $S_{\epsilon,u}:=c(\partial_r)\partial_r+\Omega(D^{\mathcal{F}_{\partial}}+\dot{\theta}(u-\Pi_{\epsilon}D^{\mathcal{F}_{\partial}}))$ on the infinite cylinder $\partial X_{0}\times \R$ equipped with the product foliation $\mathcal{F}_{\partial}\times \R$. Notice that, due to the presence of $\dot{\theta}$ this is a slightly different form of the operator \eqref{sconfronta}.  Choose some function $\psi_1$ supported in $\partial X_0\times [-1,5]$ and ${\psi_1}_{|\partial X_0 \times [0,4]}=1$. The first fact we show is
$\lim_{s\rightarrow 0}\stru(\psi_1 (e^{-sS^2_{\epsilon,u}}-e^{-sS^2_{0,u}})\psi_1)=0.$ \\ \noindent Now, $S_{\epsilon,u}=S_{0,u}-\Omega \Pi_{\epsilon}D^{\mathcal{F}_{\partial}}=c(\partial_r)\partial_r+H$ with $H=\Omega D^{\mathcal{F}_{\partial}}+\Omega \dot{\theta}u$ hence 
\begin{equation}\label{straccia}
\essp-\esspo=-\Phi \dotto \piep \deffo-2(\deffo+\dotto u)(\dotto \piep \deffo)+(\ome \dotto \piep \deffo)^2.
\end{equation}Apply the Duhamel formula $$ 
|\stru(\psi_1(\essp-\esspo)\psi_1|=
\Big|\int_0^s \stru(\psi_1^2 \Pi_{\epsilon})e^{-\delta S_{0,u}^2}(S^2_{\epsilon,u}-S^2_{0,u})\Pi_{\epsilon}e^{-(s-\delta)S^2_{\epsilon,u}}d\delta\Big|.
$$ \noindent Again from the Cheeger--Gromov relative estimates \eqref{1333}

$$|\tru(\psi_1 e^{-\delta S_u^2}\Pi_{\epsilon}\psi_1)|\leq C {\delta^{-1/2}},\quad 
\|(S^2_{\epsilon,u}-S^2_{0,u})\Pi_{\epsilon}e^{-(s-\delta)S^2_{\epsilon,u}}\|\leq C(s-\delta)^{-1/2}$$ with the constants independent from $|u|<\epsilon.$
Then the integral of the supertrace $\eqref{straccia}$ can be estimated by
the function of s,  $h(s)=C\int^s_{0}(s-\delta)^{-1/2}\delta^{-1/2}d\delta\longrightarrow_{s\rightarrow 0}0.$ In fact first split the integral into $\int_{0}^{s/2}+\int_{s/2}^s$ to prove finiteness then use the absolutely continuity of the integral for convergence to zero.
\noindent Now from the limit $\lim_{s\rightarrow 0}\stru(\psi_1 (e^{-sS^2_{\epsilon,u}}-e^{-sS^2_{0,u}})\psi_1)=0$ and the comparison argument we get that the asymptotic expansion for $s\rightarrow 0$ of $\operatorname{str}_{\Lambda}(\phi_ke^{-sD^2_{\epsilon,u}\phi_k})$ is the same of the comparison operator 
$S_{0,u}=\underbrace{c(\partial_r)\partial_r+\Omega D^{\mathcal{F}_{\partial}}}_{D}+\underbrace{\quad \quad  \quad \dot{\vartheta}u\Omega \quad \,\quad \quad \,}_{\textrm{bounded perturbation}}$ 
on the infinite cylinder. This is a very simple $u$--family of generalized laplacians (see \cite{BeGeVe} Chapter 2.7) and the
Duhamel formula
$e^{-tS^2_{0,u}}-e^{-tS^2_{0,0}}=-\int_0^ut\dot{\vartheta}\Omega e^{-t S_{0,v}}dv ds $
shows what is written in the statement i.e. $\operatorname{str}^E([e^{-sD^2_{\epsilon,u,z}}])_{(z,z)}\simeq \sum_{j\in \mathbb{N}}a_j(S_{0,u})_{(z)}s^{(j-\operatorname{dim}{\mathcal{F}})/2}$
 where the coefficients 
$a_j(S_{0,u})$ depend smoothly on $u$ and satisfy 
$a_j(S_{0,u})=0$ for $j\leq \operatorname{dim}\mathcal{F}/2$. One can take for $g$ the function 
\noindent $g(u):=\sum_{j=0}^{ \operatorname{dim}{\mathcal{F}}/2}\int_{\partial X_0 \times [0,4]}   a_j(S_{0,u})_{(z)}s^{(j-\operatorname{dim}{\mathcal{F}})/2}d\Lambda_g.$

 {\bf{3.}} This is done again by comparison with $S_{\epsilon,u}$. Consider the $r$--depending family of tangential tangential measures $(y,r)\in \partial X_0 \times [a,b] \longmapsto \operatorname{str}^E{e^{-sD^2_{\epsilon,u,(x,r)}}dxdr}$ where $x\in L_{(y,r)}$, once coupled with $d\Lambda$ it gives the measure $\mu:=\operatorname{str}^E{e^{-sD^2_{\epsilon,u,(x,r)}}dxdr}\cdot d\Lambda$ on $X$.
The Fubini theorem can certainly used during the integration process to find out that the mass of $\mu$ can be computed integrating first the $r$--depending tangential measures $y\longmapsto \operatorname{str}^E{e^{-sD^2_{\epsilon,u,(y,r)}}dy}$ against $\Lambda$ on the foliation at infinity $(\partial X_0,\mathcal{F}_{\partial})$ then the resulting function of $r$ on $[a,b],$
$
\operatorname{LIM}_{s\rightarrow 0}\int_{\partial X_0 \times [a,b]}d\mu=
\operatorname{LIM}_{s\rightarrow 0}\int_a^b \int_{\partial X_0}\operatorname{str}^E([e^{-sS^2_{\epsilon,u}}])_{(y,r),(y,r)})dy\cdot d\Lambda dx$ and this is equal to $\operatorname{LIM}_{s\rightarrow 0}\dfrac{b-a}{\sqrt{4\pi s}}\stru(e^{-s(D_{\epsilon,u}^{\mathcal{F}_{\partial}})^2})=0
$
in fact the boundary operator 
$D_{\epsilon,u}^{\mathcal{F}_{\partial}}$ is invertible and the well--known Mc--Kean--Singer formula for foliations on compact ambient manifolds (formula (7.39) in \cite{MoSc})
 says that 
$\operatorname{ind}_{\Lambda}(D_{\epsilon,u}^{\mathcal{F}_{\partial}})=\stru{e^{-s      
(D_{\epsilon,u}^{\mathcal{F}_{\partial}})^2}}$ independently from $s$.
\end{proof} 
\noindent Finally \eqref{passagg} becomes 
\begin{equation}\label{sedici}
\operatorname{ind}_{\Lambda}(D_{\epsilon,u}^+)=\langle \widehat{A}(X)\operatorname{Ch}(E/S),C_{\Lambda}\rangle-1/2 \eta_{\Lambda}(D_{\epsilon,u}^{\mathcal{F}_{\partial}})+g(u).
\end{equation}
\begin{thm}
The Dirac operator has finite dimensional $L^2-\Lambda$--index and the following formula holds
\begin{align}
\label{2111}\operatorname{ind}_{L^2,\Lambda}(D^+)=\langle\widehat{A}(X)\operatorname{Ch}(E/S),[C_{\Lambda}]\rangle +1/2[\eta_{\Lambda}(D^{\mathcal{F}_{\partial}})-h^+_{\Lambda}+h^{-}_{\Lambda}]\end{align} where
\begin{equation}\label{1001}
h^{\pm}_{\Lambda}:=\operatorname{dim}_{\Lambda}(\operatorname{Ext}(D^{\pm})-\operatorname{dim}_{\Lambda}(\operatorname{Ker}_{L^2}(D^{\pm})\end{equation} 
with the dimension of the space of extended solutions as defined in the definition \ref{170} i.e.
$\operatorname{dim}_{\Lambda}\operatorname{Ext}(D^{\pm}):=\operatorname{dim}_{\Lambda} \overline{\operatorname{Ext}(D^{\pm})}^{e^{u\theta}L^2}$ independently from small $u>0$.
\end{thm}\begin{proof}
Start from 
\begin{equation}\label{234}\operatorname{ind}_{L^2,\Lambda}(D^+_{\epsilon})=\lim_{u\downarrow 0}1/2\{\operatorname{ind}_{\Lambda}(D^+_{\epsilon,u})+\operatorname{ind}_{\Lambda}(D^+_{\epsilon,-u})+h^{-}_{\Lambda,\epsilon}-h^+_{\Lambda,\epsilon}\},\end{equation} here $h^{\pm}_{\Lambda,\epsilon}=\operatorname{dim}_{\Lambda}(\operatorname{Ext}(D^{\pm}_{\epsilon}))-\operatorname{dim}_{\Lambda}(\operatorname{Ker}_{L^2}(D^{\pm}_{\epsilon})).$ For now proposition \ref{exx} says that 
$\operatorname{Ext}(D^{\pm}_{\epsilon})=\operatorname{Ker}_{L^2}(D^{\pm}_{\epsilon,\pm})=\operatorname{Ker}_{e^{u\theta}L^2}(D^{\pm}_{\epsilon}).$
Use the identity
\eqref{sedici}
into \eqref{234} and pass to the $u\longrightarrow 0$ limit taking into account proposition \ref{0348}
$$\operatorname{ind}_{L^2,\Lambda}(D^+_{\epsilon})=\langle \widehat{A}(X)\operatorname{Ch}(E/S),C_{\Lambda}\rangle
+\dfrac{h^{-}_{\Lambda,\epsilon}-h^{+}_{\Lambda,\epsilon}}{2}+\dfrac{\eta_{\Lambda}(D_{\epsilon}^{\mathcal{F}_{\partial}})}{2}.$$
It remains to pass to the $\epsilon$--limit remembering that:
$\lim_{\epsilon\downarrow 0}\operatorname{ind}_{L^2,\Lambda}(D^+_{\epsilon})=\operatorname{ind}_{L^2,\Lambda}(D^+)$ by proposition \ref{hj}, $\lim_{\epsilon \downarrow 0}h^{-}_{\Lambda,\epsilon}-h^{+}_{\Lambda,\epsilon}=h^--h^+$ again by proposition \ref{hj}
and \\ \noindent $\lim_{\epsilon \downarrow 0}\eta_{\Lambda}(D_{\epsilon}^{\mathcal{F}_{\partial}})=\eta_{\Lambda}(D^{\mathcal{F}_{\partial}})$ by proposition \ref{0348}.
\end{proof}

\section{Comparison with Ramachandran index formula}
The Ramachandran index formula \cite{Rama} stands into index theory for foliations exactly as the Atiyah--Patodi--Singer formula stays in the classical theory. Our formula corresponds to the A.P.S cylindrical point of view. In this section we prove that the two formulas are compatible and we do it exactly in the way it is done for the single leaf case by APS. First we recall the Ramachandran Theorem
\subsection{The Ramachandran index}
\noindent Since we have chosen an opposite orientation for the boundary foliation the Ramachandran index formula here written differs from the original in \cite{Rama} exactly for its sign (as in section \ref{aaps} for the APS formula). So let us consider the Dirac operator builded in section 
\ref{geom} but acting only on the foliation restricted to the compact manifold with boundary $X_0$.
To be precise with the notation let us call $\mathcal{F}_0$ the foliation restricted to $X_0$ with leaves $\{L_x^{0}\}_x$, equivalence relation $\mathcal{R}_0$ and $D^{\mathcal{F}_0}$ the Dirac operator
acting on the field of Hilbert spaces 
$\{L^2(L_x^{0};E)\}_{x\in X_0}$. Near the boundary $D^{\mathcal{F}_0}=\left(\begin{array}{cc}0 & D^{\mathcal{F}_0^-} \\D^{\mathcal{F}_0^+} & 0\end{array}\right)=\left(\begin{array}{cc}0 & -\partial_r+D^{\mathcal{F}_{\partial}} \\ 
 \partial_r+D^{\mathcal{F}_{\partial}}
& 0\end{array}\right)$ with the boundary operator $D^{\mathcal{F}_{\partial}}.$ Let us consider the field of APS boundary conditions   $$B=\left(\begin{array}{cc}\chi_{[0,\ty)}(D^{\mathcal{F}_{\partial}}) & 0 \\0 & \chi_{(-\ty,0)}(D^{\mathcal{F}_{\partial}})\end{array}\right)=\left(\begin{array}{cc}P & 0 \\0 & \operatorname{I}-P\end{array}\right)$$ acting on the boundary foliation. In the order of ideas of the paper by Ramachandran this is a \underline{self adjoint boundary condition} i.e. its interacts with the Dirac operator in the following way:

{\bf{1.}} $B$ is a field of bounded self--adjoint operators with $\sigma B+B\sigma=\sigma$ where $\sigma$ is Clifford multiplication by the unit (interior) normal.

{\bf{2.}} If $b$ is the operator of restriction to the boundary then $(s_1,D^{\mathcal{F}_0}s_2)=(D^{\mathcal{F}_0} s_1,s_2)$ for every couple of smooth sections $s_1$ and $s_2$ such that $Bbs_1=0$ and $Bbs_2=0$.  

\noindent Next Ramachandran proves using the Browder--G{\aa}rding expansion that there's a field of restriction operators
$H^k(X_0;E)\longrightarrow H^{k-1/2}(X_0;E)$ extending $b$ where the Sobolev spaces are defined taking into account the boundary. More precisely for a leaf $L^0_x$, the space $H^k(L_x^{0};E)$ is the completion of $C^{\ty}_{c}(L_x^{0};E)$ (support possibly touching the boundary) under the usual $L^2$--based Sobolev norms. It follows from the restriction theorem that one can define the domain of $D$ with boundary condition $B$ as $H^{\ty}(X_0;E,B):=\{s \in H^{\ty}(X_0;E):Bbs=0\}.$
\begin{thm}\label{1122}(Ramachandran \cite{Rama})
The family of unbounded operators $D$ with domain $H^{\ty}(X_0;E,B)$ is essentially self--adjoint and Breuer--Fredholm in the Von Neumann algebra of the foliation with finite 
$\Lambda$--index  $\operatorname{ind}_{\Lambda}(D^{\mathcal{F}_0})=$ given by the formula
\begin{align}\nonumber\operatorname{ind}_{\Lambda}(D^{\mathcal{F}_0})&=\operatorname{dim}_{\Lambda}(\operatorname{Ker}(D^{\mathcal{F}_0^{+}}))-\operatorname{dim}_{\Lambda}(\operatorname{Ker}(D^{\mathcal{F}_0^-}))\\ \label{232}&=
\langle\widehat{A}(X)\operatorname{Ch}(E/S),C_{\Lambda}\rangle +1/2[\eta_{\Lambda}(D^{\mathcal{F}}_0) -h]\end{align} 
\end{thm}
\noindent Now we are going to prove compatibility between formula \eqref{232} and \eqref{2111}. First of all we have to relate the two Von Neumann algebras in play. Denote (according to our notation) with 
$\operatorname{End}_{\mathcal{R}_0}(E)$ the space of intertwining operators of the representation of 
$\mathcal{R}_0$ on 
$L^2(E)$ and, only in this section 
$\operatorname{End}_{\mathcal{R}_0,\Lambda}(E)$ the resulting Von Neumann algebra with trace 
$\operatorname{tr}_{\mathcal{R}_0,\Lambda}$ in order to make distinction from 
$\operatorname{End}_{\mathcal{R},\Lambda}(E)$ the Von Neumann algebra of random operators associated with the representation of 
$\mathcal{R}$. Start with a measurable fields of bounded operators 
$X_0 \ni B_x\longmapsto B_x:L^2(L_x^{0};E)\longrightarrow L^2(L_x^{0};E)$ with $B_x=B_y$ a.e. if 
$(x,y)\in \mathcal{R}_0$. There's a natural way to extend $B$ to a field of operators in 
$\operatorname{End}_{\mathcal{R}}(E)$. 

{\bf{1.}} If $x\in X_0$ simply let 
$\imath B_x$ act to $L^2(L_x;E)$ to be zero on the cylinder
$$\imath B_x:L^2(L_x^0;E)\oplus L^2(\partial L_x^0\times (0,\ty);E)\longrightarrow L^2(L_x^0;E)\oplus L^2(\partial L_x^0\times (0,\ty);E)$$
$\imath B_x(s,t):=(B_x s,0). $

{\bf{2.}} If $x\in \partial X_0\times (0,\ty)$ define $\imath B_x:=\imath B_{p(x)}$ where $p:\partial X_0\times (0,\ty)\longrightarrow \partial X_0$ is the base projection and $\imath B_{p(x)}$ is defined by point $1.$

\begin{prop}\label{987}
The map $\imath:\operatorname{End}_{\mathcal{R}_0}(E)\longrightarrow \operatorname{End}_{\mathcal{R}}(E)$ as defined above passes to the quotient to an injection 
$\imath:\operatorname{End}_{\mathcal{R}_0,\Lambda}(E)\longrightarrow \operatorname{End}_{\mathcal{R},\Lambda}(E)$ between the Von Neumannn algebras of Random operators preserving the two natural traces
$\operatorname{tr}_{\mathcal{R},\Lambda}(\imath B)=\operatorname{tr}_{\mathcal{R}_0,\Lambda}(B).$
\end{prop}
\begin{proof}The first part is clear. An intertwining operator $B=\{B_x\}_{x\in X_0}$ is zero $\Lambda$--a.e. in $X_0$ then also does $\imath B$ in $X$ for any transversal $T$ contained in the cylinder can slide by holonomy to a transversal contained in $X_0$. About the identity on traces remember the link between the direct integral algebras and the algebras of random operators i.e. Lemme 8 pag 48 in \cite{Cos}. Choose $\nu$ to be the longitudinal Riemannnian metric then $\Lambda_{\nu}$ is the integration of $\nu$ against $\Lambda$. Let $P_0$ be the Von Neumann algebra of $\Lambda_{\nu}$--a.e. classes of measurable fields of operators $X_0\ni x \longmapsto B_x\in B(L^2(L_x^0;E))$ and $P$ the corresponding algebra builded replacing $X_0$ with $X$ and $B(L^2(L_x^0;E))$ with $B(L^2(L_x^0;E))$. Pass to a ultraweak dense ideal of operarators such that the corresponding family $X\ni y\longmapsto \int \imath B_x d\nu^y$ is bounded. Then Lemme 8 pag 48 in \cite{Cos} says that 
$\operatorname{tr}_{\mathcal{R},\Lambda}(\imath B)=\int_{X}\operatorname{Trace}(B_x)d\Lambda_{\nu}(x)=\int_{X_0}\operatorname{Trace}(B_x)d\Lambda_{\nu}(x)=\operatorname{tr}_{\mathcal{R}_0,\Lambda}(B).$
\end{proof}
 \begin{thm}
 Let $\operatorname{Pr}\operatorname{Ker}(D^{\mathcal{F}_0^{\pm}}) \in \operatorname{End}_{\mathcal{R}_0,\Lambda}(E)$ the projection on the Kernel of $D^{\mathcal{F}_0^{\pm}}$ with domain given by the boundary condition $Px=0,\,(\operatorname{I}-P=0)$ as in the formula of Ramachandran.
 Let also $\operatorname{Pr}\operatorname{Ker}_{L^2}(D^{\pm}) \in \operatorname{End}_{\mathcal{R},\Lambda}(E)$ 
 be the projection on the $L^2$--kernel of the leafwise operator on the foliation with the cylinder attached and 
 $\operatorname{Pr}\overline{\operatorname{Ext}(D^{\pm})} \in\operatorname{End}_{\mathcal{R},\Lambda}(e^{u\theta}L^2 E)$ 
 be the projection on the closure of the space of extended solution seen in 
 $e^{u\theta}$ for sufficiently small positive $u$.
 \begin{enumerate}
\item $\imath \operatorname{Pr}\kepp$ is (Murray--Von Neumann $\sim$) equivalent to $\operatorname{Pr}\operatorname{Ker}_{L^2}(D^+)$ in $\operatorname{End}_{\mathcal{R},\Lambda}(E)$ i.e. there exists a partial isometry $u\in \operatorname{End}_{\mathcal{R},\Lambda}(E)$ such that $u^*u=\imath \operatorname{Pr}\kepp$ and $uu^*=\operatorname{Pr}\operatorname{Ker}_{L^2}(D^+).$ In particular
$\operatorname{dim}_{\mathcal{R}_0,\Lambda}\kepp=\operatorname{dim}_{\mathcal{R},\Lambda} \operatorname{Ker}_{L^2}(D^+).$
\item
$\imath \operatorname{Pr}\operatorname{Ker}_{L^2}(D^{\mathcal{F}_0^-})
\sim \operatorname{Pr} \overline{\operatorname{Ext}(D^-)}^{e^{u\theta}L^2},$ 
for sufficiently small $u$ and equivalence in 

\noindent $\operatorname{End}_{\Lambda}(e^{u\theta}L^2(E))$ 
with the inclusion $\imath:\operatorname{End}_{\mathcal{R}_0,\Lambda}(E)\longrightarrow    \operatorname{End}_{\Lambda}(e^{u\theta}L^2(E))  $ defined as in proposition \ref{987}. 
As a consequence $\operatorname{dim}_{\Lambda}\operatorname{Ker}(D^{\mathcal{F}_0^-})=\operatorname{dim}_{\Lambda}\operatorname{Ext}(D^-).$
\end{enumerate}
 \end{thm}
 \begin{proof}The idea is contained in A.P.S.
  \cite{AtPaSi1} when they prove the equivalence between the boundary value problem and the $L^2$ cylindrical problem. Their main instrument is the eigenfunction expansion of the operator at the boundary; now we use the Browder--Garding generalized expansion to see that any solution of the boundary value problems extends to a solution of the operator on the cylinder. 

 {\bf{1.}} Use the Browder--G{\aa}rding expansion as in the proof of the finiteness of the projection on the kernel \ref{2344}. For a single leaf, the isomorphism
 $$L^2(\partial L_x^0\times (-1,0])\longrightarrow \bigoplus_{j\in \mathbb{N}}L^2(\R,\mu_j)\otimes L^2((-1,0]) $$ represents a solution of the boundary value problem as $h_{j}(r,\lambda)=\chi_{(-\ty,0)}(\lambda)e^{-\lambda r}h_{j0}(r)$ hence the solution can be extended to the cylinder of the leaf $\partial L_x^0\times (0,\ty)$. This clearly gives a field of linear isomorphisms $T_x:\operatorname{Ker}(D_x^{\mathcal{F}_0^+})\longrightarrow \operatorname{Ker}_{L^2}(D_x^+)$ for $x\in X_0$. First extend $T_x$ to all $L^2(L_x^0;E)$ to be zero on $ \operatorname{Ker}(D^{\mathcal{F}_0^+})^{\bot}$ then let $x$ take values also in $X$ according to the method explained before i.e. put $T_x:=T_{p(x)}$ for $x$ in the cylinder. Take the polar decomposition $T_x=u_x|T_x|$, then $u_x$ is a partial isometry with initial space $\operatorname{Ker}(D_x^{\mathcal{F}_0^+})$ and range $\operatorname{Ker}(D_x^+),$ i.e 
 $$u_x^*u_x=\operatorname{Pr} \operatorname{Ker}(D_x^{\mathcal{F}_0^+}),\quad u_xu_x^*=\operatorname{Pr} \operatorname{Ker}(D_x^+).$$
 We have to look at this relation into the Von Neumann algebra of the foliation on $X$. Split every $L^2$ space of the leaves as $L^2(L_{p(x)}^0;E)\oplus L^2(\partial L_{p(x)}^0\times (0,\ty);E)$. With respect to the splitting, forgetting the indexes $x$ downstairs, we have $u=\left(\begin{array}{cc}u_{11} & 0 \\u_{21} & 0\end{array}\right)$ acting on the field of $L^2(X;E)$ spaces of the leaves. Then $u^*=\left(\begin{array}{cc}u_{11}^* & u_{21}^* \\0 & 0\end{array}\right)$ with conditions $u_{11}u_{21}^*=0$ and $u_{21}u_{11}^*=0.$
 Finally $uu^*=\left(\begin{array}{cc}u_{11}u_{11}^*+ u_{21}u_{21}^*& 0 \\0 & 0\end{array}\right)
 \left(\begin{array}{cc}\operatorname{Pr}(D^{\mathcal{F}_0^+}) & 0 \\0 & 0\end{array}\right)=\imath \operatorname{Pr}(D^{\mathcal{F}_0^+})$ and similarly $u^*u=\operatorname{Pr}(D^+).$
 
 {\bf{2.}} It is very similar to statement 1. in fact writing the Browder--G{\aa}rding expansion and imposing the adjoint boundary condition one ends directly into the space of the extended solutions.
\end{proof}

\noindent To conclude now we can compare Ramachandran index with ours; let's compare formula \eqref{232} with \eqref{2111} keeping in mind that, the index of Ramachandran is now our extended index (see section \ref{aaps} )
 $\operatorname{ind}_{\Lambda}(D^{\mathcal{F}_0})={\operatorname{ind}_{\Lambda,L^2}(D^+)}=\operatorname{dim}_{\Lambda} \operatorname{Ker}_{L^2}(D^+)-\operatorname{Ker}_{L^2}(D^-)$ to obtain the equation
 $\operatorname{dim}_{\Lambda} \operatorname{Ext}(D^-)-\operatorname{dim}_{\Lambda} \operatorname{Ker}_{L^2}(D^-)=(h^{-}_{\Lambda}-h^{+}_{\Lambda})/2+h/2.$
The same argument applied to the (formal) adjoint of $D^+$ leads to the equation

\noindent $\operatorname{dim}_{\Lambda} \operatorname{Ext}(D^+)-\operatorname{dim}_{\Lambda} \operatorname{Ker}_{L^2}(D^+)=(h^{+}_{\Lambda}-h^{-}_{\Lambda})/2+h/2,$ then 
$$h=h^{+}_{\Lambda}+h^{-}_{\Lambda}$$ exactly as in the Atiyah Patodi Singer paper.


\small
\begin{thebibliography}{alpha}





\bibitem{AtBo}
M.~F. Atiyah., R. Bott, and V.~K. Patodi.
\newblock On the heat equation and the index theorem
\newblock {Invent. Math. (19)}, 279--330 1973. errata, No. ibid. (28) 277--280 1975.

\bibitem{AB}
M.~F. Atiyah., R. Bott.
\newblock The index theorem for manifolds with boundary.
\newblock {Differential analysis}, Bombay colloquium (Oxford 1964).


\bibitem{AtPaSi1}
M.~F. Atiyah, V.~K. Patodi, and I.~M. Singer.
\newblock Spectral asymmetry and {R}iemannian geometry. {I}.
\newblock {\em Math. Proc. Cambridge Philos. Soc.}, 77:43--69, 1975.







\bibitem{BeGeVe}
Nicole Berline, Ezra Getzler, and Mich{\`e}le Vergne.
\newblock {\em Heat kernels and {D}irac operators}, volume 298 of {\em
  Grundlehren der Mathematischen Wissenschaften [Fundamental Principles of
  Mathematical Sciences]}.
\newblock Springer-Verlag, Berlin, 1992.




\bibitem{Bifree}
Jean-Michel Bismut and D. S. Freed.
\newblock The Analysis of elliptic families.
 {\em Comm. Math. Phys.}, volume~107, 103--163. 1986.

\bibitem{lesch}
Bernhelm Booss-Bavnbek, Matthias Lesch, Chaofeng Zhu
\newblock  The Calderon Projection: New Definition and Applications 
{\em arXiv:0803.4160v1}

\bibitem{Br2}Manfred Breuer 
\newblock Fredholm theories on Von Neumann algebras II
\newblock  {\em Math. Ann. } 180, 313--325. 1969.

\bibitem{cp0} A.L. Carey \& J. Phillips 
\newblock{Unbounded Fredholm modules and spectral flow}.
\newblock {\em Can. J. Math} 50(4):673--718 1998

\bibitem{cp} A.L. Carey \& J. Phillips \& A. Rennie \& F.A. Sukochev
\newblock{ ``The local index formula in semifinite von
Neumann algebras II: the even case''}.
\newblock {\em Adv. Math.} 202:517--554 2006








\bibitem{ChGrTa}
Jeff Cheeger, Mikhail Gromov, and Michael Taylor.
\newblock Finite propagation speed, kernel estimates for functions of the
  {L}aplace operator, and the geometry of complete {R}iemannian manifolds.
\newblock {\em J. Differential Geom.}, 17(1):15--53, 1982.

\bibitem{Ch}
 P. R. Chernoff
\newblock Essential self--adjointness of powers of generators of hyperbolic equations.
\newblock {\em J. Functional Analysis.}, 12:401--414, 1973.


\bibitem{Cos}
Alain Connes.
\newblock Sur la th\'eorie non commutative de l'int\'egration.
\newblock In {\em Alg\`ebres d'op\'erateurs (S\'em., Les Plans-sur-Bex, 1978)},
  volume 725 of {\em Lecture Notes in Math.}, pages 19--143. Springer, Berlin,
  1979.




\bibitem{Di}
J. Dieudonne.
\newblock {Treatise on modern analysis}. Vol 10-VII
\newblock Academic Press, New York, 1988.

\bibitem{Dix}
J. Dixmier.
\newblock {Von Neumann algebras}. 
\newblock North Holland, Amsterdam, 1988.


\bibitem{DuSc}
N. Dunford and J. Schwartz 
\newblock Linear operators part 2, Spectral theory.
\newblock Interscience, New York, 1963

\bibitem{Ge}
Ezra Getzler.
\newblock Cyclic homology and the {A}tiyah-{P}atodi-{S}inger index theorem.
\newblock In {\em Index theory and operator algebras (Boulder, CO, 1991)},
  volume 148 of {\em Contemp. Math.}, pages 19--45. Amer. Math. Soc.,
  Providence, RI, 1993.




\bibitem{hl}
James Heitsch and Conor Lazarov.
\newblock Homotopy invariance of foliation Betti numbers,
\newblock In {Invent. Math}, 104 pages 321--347 1991.








\bibitem{koo}
Yu. A. Kordyukov. 
\newblock Functional calculus for tangentially elliptic operators on foliated manifolds. \newblock {\em Analysis and Geometry of foliated Manifolds, Proceedings of the VII International Colloquim on Differential geometry, Santiago de Compostela, 1994.} 
\newblock Singapore World scientific, 1995, 113-136.







\bibitem{Me}
Richard~B. Melrose.
\newblock {\em The {A}tiyah-{P}atodi-{S}inger index theorem}, volume~4 of {\em
  Research Notes in Mathematics}.
\newblock A K Peters Ltd., Wellesley, MA, 1993.




\bibitem{MoSc}
Calvin~C. Moore and Claude Schochet.
\newblock {\em Global analysis on foliated spaces}, volume~9 of {\em
  Mathematical Sciences Research Institute Publications}.
\newblock Springer-Verlag, New York, 1988.
\newblock With appendices by S. Hurder, Moore, Schochet and Robert J. Zimmer.



\bibitem{nus}
A. E. Nussbaum 
\newblock Reduction theory for unbounded closed operators in hilbert space,
\newblock {\em Duke Math. Journal}, 31(1):33--44, 1964.




\bibitem{Peric}
Goran Peric.
\newblock Eta invariants of Dirac operators on foliated manifolds,
\newblock {\em Trans. Amer. Math. Soc.} 334(2): 761--782. 1992



\bibitem{Rama}
Mohan Ramachandran.
\newblock Von {N}eumann index theorems for manifolds with boundary.
\newblock {\em J. Differential Geom.}, 38(2):315--349, 1993.

\bibitem{revuz}
D. Revuz
\newblock Markov chains.
 \newblock North-Holland Mathematical Library, vol. 11, North-Holland, Amsterdam; American Elsevier, New York, 1975, x + 336 pp., 35.50

\bibitem{Reed}
M Reed and B. Simon
\newblock Methods of Mathematical Physics IV.
\newblock {Academic Press}, New York 1978.




\bibitem{Roel}
John Roe.
\newblock Elliptic operators, topology and asymptotic methods.
\newblock {\em Pitman research notes in Mathematics}, 179 1988.



\bibitem{Roeff}
\newblock Finite propagation speed and Connes' foliation algebra.
\newblock Math. Proc. Cambridge Philos. Soc. 102 (1987), no. 3, 459--466. 





\bibitem{take}
Masamichi Takesaki
\newblock Theory of operator algebra I
\newblock Springer Verlag






\bibitem{Vai}
B. Vaillant. 
\newblock Indextheorie fur Uberlagerungen. 
\newblock Diplomarbeit, Universitat. Bonn, 
http://styx.math.uni-bonn.de/boris/diplom.html, 1997

\bibitem{Va}
Stephane Vassout.
\newblock {\em Feuilletage et résidu non commutatif longitudinal}.
\newblock PhD thesis. Paris, 2001.
\end{thebibliography}
\end{document}